\documentclass[11pt,twoside,a4paper]{amsart}
\usepackage{amssymb}
\usepackage{amsmath,amscd}
\usepackage[initials,nobysame,shortalphabetic]{amsrefs}
\usepackage[all,cmtip]{xy}

\usepackage[utf8]{inputenc}

\usepackage{geometry}
\def\End{{\rm End}}
\newcommand\cl[1]{\left[#1\right]}

\def\deg{\text{deg}\,}

\def\dbar{{\bar\partial}}
\def\C{{\mathbb C}}

\def\Np{{\mathbb N}}

\def\E{{\mathcal E}}
\def\Ext{{\mathcal Ext}}

\def\PM{{\mathcal{PM}}}
\def\Hom{{\rm Hom\, }}
\def\SHom{{\mathcal{H}om\, }}
\def\codim{{\rm codim\,}}

\def\Ok{{\mathcal O}}
\def\CH{\mathcal{CH}}
\DeclareMathOperator{\Id}{Id}
\DeclareMathOperator{\supp}{supp}

\DeclareMathOperator{\rank}{rank}
\DeclareMathOperator{\im}{im}
\DeclareMathOperator{\coker}{coker}
\DeclareMathOperator{\depth}{depth}

\DeclareMathOperator{\tr}{tr}
\DeclareMathOperator{\res}{res}

\def\be{\begin{equation}}
\def\ee{\end{equation}}

\newtheorem{thm}{Theorem}[section]
\newtheorem{lma}[thm]{Lemma}

\newtheorem{prop}[thm]{Proposition}

\theoremstyle{definition}

\newtheorem{df}[thm]{Definition}

\theoremstyle{remark}

\newtheorem{preremark}[thm]{Remark}
\newtheorem{preex}[thm]{Example}

\newenvironment{remark}{\begin{preremark}}{\end{preremark}}
\newenvironment{ex}{\begin{preex}}{\end{preex}}

\numberwithin{equation}{section}

\begin{document}

\title{Residue currents and fundamental cycles}

\date{\today}

\author{Richard L\"ark\"ang \& Elizabeth Wulcan}

\address{Mathematical Sciences\\Chalmers University of Technology and the University of Gothenburg\\S-412 96 
Gothenburg\\SWEDEN}

\email{larkang@chalmers.se, wulcan@chalmers.se}

\thanks{The authors were supported by the Swedish Research Council.}

%\subjclass{}

\keywords{}

\begin{abstract}
    We give a factorization of the fundamental cycle of an analytic
    space in terms of certain differential
    forms and residue currents associated with a locally free resolution
    of its structure sheaf. Our result can be seen as a generalization of the classical 
    Poincar\'e-Lelong formula. 
    It is also a current version of a result by Lejeune-Jalabert,
    who similarly expressed the fundamental class of a
    Cohen-Macaulay analytic space in terms
    of differential forms and cohomological residues. 
\end{abstract}

\maketitle

\section{Introduction}

Given a holomorphic function $f$ on a complex manifold $X$, recall
that the classical \emph{Poincar\'e-Lelong formula} asserts that
$\dbar \partial \log |f|^2 = 2\pi i[Z]$, 
where $[Z]$ is the current of integration (or Lelong current) of the divisor $Z$ of $f$
counted with multiplicities,  
or, more precisely, (the current of integration of) 
the fundamental cycle of $Z$. 
Formally we can rewrite the Poincar\'e-Lelong formula 
as 
\begin{equation}\label{poin}
\frac{1}{2\pi i} \dbar\frac{1}{f} \wedge df = [Z]. 
\end{equation} 
This factorization of $[Z]$ 
can be made rigorous if we construe $\dbar(1/f)$
as the \emph{residue current} of $1/f$, 
introduced by Dolbeault, \cite{Dol}, and Herrera and Lieberman, \cite{HL}, 
and defined, e.g., as 
\begin{equation}\label{martes}
\lim_{\epsilon\to 0}\dbar\chi(|f|^2/\epsilon)\frac{1}{f},
\end{equation}
where $\chi(t)$ is (a smooth approximand of) the characteristic function of the
interval $[1,\infty)$. 
The current $\dbar(1/f)$ satisfies that a holomorphic
function $g$ on $X$ is in the ideal (sheaf) $\mathcal J (f)$ generated by $f$ if and only if $g\dbar
(1/f)=0$. This is referred to as the \emph{duality principle} and it is
central to many applications of residue currents; in a way
$\dbar(1/f)$ can be thought of as a current representation of the
ideal $\mathcal J (f)$. 
In this paper we prove that (the current of integration along) the fundamental cycle of any analytic
space admits a natural factorization as a smooth ``Jacobian'' 
factor times a residue current, analogous to \eqref{poin}. 

Let $Z\subset X$ be a (not necessarily reduced) analytic space. 
The \emph{fundamental cycle} of $Z$, seen as a current on $X$, is the current
\begin{equation} \label{eqfundcycle}
    [Z] = \sum m_i [Z_i],
\end{equation}
where $Z_i$ are the irreducible components of $Z_{\rm red}$, $[Z_i]$
are the currents of integration of the (reduced) subspaces $Z_i$,
and $m_i$ are the geometric multiplicities
of $Z_i$ in $Z$.

For a generic $z\in Z_i$, $\Ok_{Z,z}$ is a free
$\Ok_{Z_i,z}$-module of constant rank. One way of defining the \emph{geometric
multiplicity} $m_i$ of $Z_i$ in $Z$ is as this rank. 
Equivalently $m_i$ can be defined as the length of the Artinian ring $\Ok_{Z,Z_i}$, see, e.g.,
\cite{Fulton}*{Chapter~1.5}. The equivalence of the two definitions can be proved with the
help of \cite{Fulton}*{Lemma~1.7.2}. 
If $Z_{\rm red} = \{ z \}$ is a point, and $Z$ is defined by an ideal
sheaf $\mathcal{J}$, i.e., $\Ok_Z=\Ok_X/\mathcal J$, 
then the geometric multiplicity of ($Z_\text{red}$ in) $Z$ is $\dim_{\C} \Ok_{X,z}/\mathcal{J}_z$.
If $\dim Z_i > 0$, then for generic $z \in Z_i$ and $H \subset X$ a complex
manifold transversal to $(Z_i,z)$, $m_i = \dim_{\C} \Ok_{X,z}/(\mathcal{J} + \mathcal{J}_H)_z$,
where $\mathcal{J}_H$ is the ideal of holomorphic functions vanishing on $H$.

\smallskip 

We will consider $Z$ such that $\Ok_Z$ has a global locally free resolution over $\Ok_X$.
Such a resolution exists for any $Z$ for example when $X$ is projective. If $X$ is Stein, then any $Z$ has a semi-global resolution,
i.e., it has a free resolution on every compact in $X$. Assume that
\begin{equation}\label{turkos}
    0 \xrightarrow[]{} E_\nu \xrightarrow[]{\varphi_\nu} E_{\nu-1} \xrightarrow[]{\varphi_{\nu-1}}
    \cdots \xrightarrow[]{\varphi_2} E_1 \xrightarrow[]{\varphi_1} E_0,
\end{equation}
is this locally free resolution, i.e., \eqref{turkos} is an exact
complex of locally free 
$\Ok_X$-modules such that $\coker \varphi_1\cong\Ok_Z$. If the 
corresponding vector bundles are equipped with Hermitian metrics we
say that $(E, \varphi)$ is a \emph{Hermitian locally free resolution}
of $\Ok_Z$ over $\Ok_X$. Given such an $(E,\varphi)$, 
in \cite{AW1} Andersson and the second author
constructed an $\End E$-valued residue current $R^E = \sum R_k^E$, where $E=\bigoplus E_k$,
and $R_k^E$ takes values in $\Hom(E_0,E_k)$. This current satisfies a duality
principle and it has found many applications; e.g., it has
been used to obtain new results on the $\dbar$-equation on singular
varieties, \cite{AS}, and a global effective Brian\c con-Skoda-Huneke
theorem, \cite{AWSemester}.

If $f$ is a holomorphic function on $X$ 
and $E_0\cong\Ok_X$ and $E_1\cong \Ok_X$ are trivial line bundles, then 
\begin{equation*}
 0 \xrightarrow[]{} \Ok_X \xrightarrow[]{\varphi_1} \Ok_X,
\end{equation*}
where $\varphi_1$ is the $1\times 1$-matrix $[f]$, gives a locally free resolution of $\mathcal
O_Z:=\mathcal O/\mathcal J(f)$. 
In this case (the coefficient of) $R^E=R^E_1$ is just $\dbar (1/f)$, and
the Poincar\'e-Lelong formula 
\eqref{poin} can be written as\footnote{The relation between the signs
  in \eqref{poin} and \eqref{snodd} is explained in Section
  ~\ref{matrisnot}.}
\begin{equation}\label{snodd}
\frac{1}{2\pi i} d\varphi_1 R^E_1=[Z]. 
\end{equation} 

Our main result is the following generalization of \eqref{snodd}. 
\begin{thm} \label{thmmain}
    Let $Z \subset X$ be an analytic space of pure codimension $p$, let $(E,\varphi)$
    be a Hermitian locally free resolution of $\Ok_Z$ over $\Ok_X$, where $\rank E_0 = 1$,
    and let $D$ be the connection\footnote{The connection $D$ is
    defined by \eqref{eqDend}.}
    on $\End E$ induced by arbitrary connections on $E_0,\dots,E_p$.
    Then
    \begin{equation} \label{eqmain}
        \frac{1}{(2\pi i)^p p!} D\varphi_1 \cdots D\varphi_p R^E_p = [Z].
    \end{equation}
\end{thm}

Note that the endomorphism $D\varphi_1 \cdots D\varphi_p$ depends on the choice of connections
on $E_0,\dots,E_p$ and the current $R^E_p$ in general depends on the choice of Hermitian
metrics on $E_0,\dots,E_p$. There is no assumption of any relation between the connections
and the Hermitian metrics.

Various special cases of Theorem~\ref{thmmain} and related results have been proved
earlier: 
Assume that $Z$ is a complete intersection of codimension $p$, 
i.e., $\mathcal O_Z=\mathcal O_X/\mathcal J$, where $\mathcal J$ is a complete
intersection ideal, generated by, say, $f = (f_1,\dots,f_p)$. Then Coleff and
Herrera proved in \cite{CH} the following generalization of the Poincar\'e-Lelong formula
\eqref{poin}:
\begin{equation} \label{eqpl}
    \frac{1}{(2\pi i)^p}\dbar \frac{1}{f_p} \wedge \cdots \wedge
    \dbar \frac{1}{f_1} \wedge df_1 \wedge \cdots \wedge df_p=[Z], 
\end{equation}
where $\dbar(1/f_p)\wedge \cdots\wedge \dbar(1/f_1)$ is the so-called \emph{Coleff-Herrera product}
of $f$. In this situation, one may choose the resolution $(E,\varphi)$ such that \eqref{eqmain}
becomes precisely \eqref{eqpl}, see \eqref{eqplcomputation}.

In \cite{DP} Demailly and Passare extended \eqref{eqpl} to the case when
$Z$ is a locally complete intersection, cf. Remark ~\ref{rem:DP}.
The result of Demailly-Passare was further extended by Andersson in \cite{ALelong},
where he proved that if one considers so-called Bochner-Martinelli residue
currents associated to generators of the defining ideal of an analytic space $Z$ of
pure dimension, and form a current similar to the left-hand side of \eqref{eqmain},
then a similar formula holds. This is a variant of the so-called King's formula, where the
right-hand side of \eqref{eqmain} is a current of integration like \eqref{eqfundcycle}, but where
the multiplicities $m_i$ are the corresponding \emph{algebraic}
(or \emph{Hilbert-Samuel}) multiplicities, see, e.g., \cite{Fulton}*{Chapter~4.3}.

If $Z$ and $(E,\varphi)$ are as in Theorem~\ref{thmmain}, then by
\cite{AndNoeth}*{Example~1}, there exists \emph{some} holomorphic 
$\Hom (E_p,E_0)$-valued form $\xi$ such that $\xi
R^E_p=[Z]$. Our Theorem ~\ref{thmmain} thus states that $(1/(2\pi i)^pp!)
D\varphi_1\cdots D\varphi_p$ is an explicit such $\xi$. 

In previous works, \cite{LW1} and \cite{WLJ}, we proved Theorem ~\ref{thmmain} for certain
resolutions of monomial ideals by explicitly computing the residue
currents $R^E$
and the Jacobian factors $D\varphi_1\cdots D\varphi_p$ respectively.

\smallskip

Another result that is closely related to ours,  although not formulated in terms of residue currents,
is a cohomological version of Theorem~\ref{thmmain} in the Cohen-Macaulay case 
due to Lejeune-Jalabert, \cite{LJ1}.
Given a free resolution $(E,\varphi)$ of
$\Ok_{Z,z}$ 
of minimal length, where $Z$ is a Cohen-Macaulay analytic
space,  
she constructed a generalization of the Grothendieck residue pairing, which in a
sense is a cohomological version of the current in \cite{AW1},  
and proved that the fundamental class of $Z$ at $z$ is represented by
$D\varphi_1\cdots D\varphi_p$. 
In Section
~\ref{sectlj} we describe this in more details and also discuss the
relation to our results. 
The relationship between Lejeune-Jalabert's residue pairing and the
residue currents in \cite{AW1} is elaborated in \cite{Lar15}, see
also \cites{Lund1, Lund2}.

\smallskip 

To be precise, the current in the left-hand side of \eqref{eqmain}
takes values in $\End E_0$. 
However, since $E_0$ has rank $1$, it is naturally identified with a
scalar-valued current. 
In fact, it is possible to drop the
assumption that $\rank E_0=1$, but to make sense of \eqref{eqmain} we
then need to turn the $\End E_0$-valued current 
\[
\Theta:= \frac{1}{(2\pi i)^p p!} D\varphi_1 \cdots D\varphi_p R^E_p
\]
into a scalar-valued current. 
We will describe two natural ways of doing this. The first one is to
take the trace $\tr \Theta$ of $\Theta$.
Secondly, let $\tau$ be the natural surjection 
$\tau : E_0 \to \coker \varphi_1 \cong \Ok_Z$. Since $R^E_p
\varphi_1 = 0$, see \eqref{eqannright} below,
one gets a well-defined $\Hom(\Ok_Z,E_p)$-valued current $R^E_p
\tau^{-1}$ by (locally) letting
$R^E_p \tau^{-1} f := R^E_p f_0$ for any section $f_0$ of $E_0$ such
that $\tau f_0 = f$. 
It follows that $\tau \Theta \tau^{-1}$ is a well-defined
$\End(\Ok_Z)$-valued current, which can be identified with a
scalar-valued current (annihilated by $\mathcal{J}$, where $\mathcal{J} \subset \Ok_X$ is the ideal defining $Z$).
Note that if $\rank E_0=1$, then $\tr \Theta$ and $\tau\Theta\tau^{-1}$
coincide with $\Theta$ (regarded as scalar currents).

\begin{thm} \label{thmmain2}
    Let $Z \subset X$ be an analytic space of pure codimension $p$, let $(E,\varphi)$
    be a Hermitian locally free resolution of $\Ok_Z$ over $\Ok_X$, and let $D$ be the connection
    on $\End E$ induced by arbitrary connections on $E_0,\dots,E_p$.
    Then
    \begin{equation} \label{eqmain2}
        \frac{1}{(2\pi i)^p p!}\tr \left( D\varphi_1 \cdots D\varphi_p R^E_p \right) = [Z]
    \end{equation}
    and
    \begin{equation} \label{eqmain3}
        \frac{1}{(2\pi i)^p p!} \tau D\varphi_1 \cdots D\varphi_p R^E_p \tau^{-1} = [Z],
    \end{equation}
    where $\tau$ is the natural surjection $\tau : E_0 \to \coker \varphi_1 \cong \Ok_Z$.
\end{thm}

In view of the discussion above, note that Theorem ~\ref{thmmain} is
just a special case of Theorem ~\ref{thmmain2}. 

\smallskip

The proof of Theorem ~\ref{thmmain2} is given in Section
~\ref{sectmainproof}. The first key ingredient is two lemmas, Lemmas
~\ref{lmaindep} and ~\ref{lmaindep2}, which assert that the left-hand sides of
\eqref{eqmain2} and \eqref{eqmain3}, respectively, only depend on $Z$
and not on the choice of $(E, \varphi)$ or $D$. In particular, it follows that the
left-hand side of \eqref{eqmain2} coincides with the left-hand side of
\eqref{eqmain3}, cf.\ \eqref{eqtraceidentification}. Thus, to
prove Theorem ~\ref{thmmain2} it is enough to prove \eqref{eqmain2} for
a specific choice of resolution and connection.  
The proofs of Lemmas~\ref{lmaindep} and ~\ref{lmaindep2}
rely on a
comparison formula for residue currents due to the first author,
\cite{LComp}, see Section ~\ref{ssectcomp}.

By the \emph{dimension principle}, Proposition ~\ref{proppmdim}, for so-called \emph{pseudomeromorphic currents},
see Section ~\ref{ssectpm}, it suffices to prove \eqref{eqmain2} generically on $Z_\text{red}$ (i.e., outside a hypersurface of
$Z_\text{red}$). For $z$ generically on $Z_\text{red}$ we can use a certain
\emph{universal free resolution} of $\Ok_{Z,z}$, based on a construction by 
Scheja and Storch, \cite{SS}, and
Eisenbud, Riemenschneider and Schreyer, \cite{ERS}; this is described
in Section ~\ref{sectkoszul}. The inspiration to use this universal
free resolution comes from \cite{LJ1}. 
The resolution is in general far from being minimal, in
particular, $\rank E_0>1$ in general, but it is explicit enough so that
we can explicitly
compute \eqref{eqmain2}, see Lemma ~\ref{lma:mainKoszul}. 

\smallskip 

In Theorems ~\ref{thmmain} and ~\ref{thmmain2} we assume that $Z$ has
pure codimension, or, equivalently, 
pure dimension. In fact, for the proofs we only need that $Z$
has \emph{pure dimension} in the weak sense that all irreducible components
of $Z_\text{red}$ have the same dimension, in other words, all
minimal primes of $\mathcal J$ have the same dimension. In particular,
we allow $\mathcal J$ to have embedded primes. 
\begin{ex}\label{midair}
Let $Z\subset \C^2$ be defined by $\mathcal{J} = \mathcal{J}(y^k,x^\ell y^m) \subset \Ok_{\C^2}$,
where $m < k$.
Then $Z$ has pure dimension, since $Z_\text{red}$ equals 
$\{y=0\}$, which is irreducible. However, note that $\mathcal J$ 
has an embedded prime $\mathcal J(x,y)$ of dimension $0$. 
\end{ex}

\begin{ex}\label{condo}
Let $Z\subset \C^3$ be defined by $\mathcal{J} = \mathcal{J}(xz,yz) \subset \Ok_{\C^3}$. 
Then $Z$ does not have pure dimension, since its irreducible components
$\{ z = 0 \}$ and $\{ x = y = 0 \}$ have dimension $2$ and $1$, respectively.
\end{ex}

We get a version of Theorem ~\ref{thmmain} also when $Z$ does not have
pure dimension, without much extra work. However, the formulation 
becomes slightly more involved. 
Since the residue currents $R^E_k$ are pseudomeromorphic, see Section
~\ref{ssectpm}, it follows that
one can give a natural meaning to the
restrictions ${\bf 1}_WR^E_k$ if $W$ is a subvariety of $X$. 

\begin{thm} \label{thmnonpure}
    Let $Z \subset X$ be an analytic space. Assume that $\dim X=N$ and $\codim Z=p$. 
    Let $(E,\varphi)$ be a Hermitian locally free resolution of $\Ok_Z$ over $\Ok_X$, where $\rank E_0 = 1$,
    and let $D$ be the connection on $\End E$ induced by arbitrary connections on $E_0,\dots,E_N$.
    Let $W_k$ be the union of the components of $Z_{\rm red}$ of codimension $k$,
    and define $R_{[k]} := {\bf 1}_{W_k} R^E_k$.
    Then
\begin{equation}\label{jesse}
    \sum_{k=p}^{N} \frac{1}{(2\pi i)^kk!} D\varphi_1 \cdots D\varphi_k
    R_{[k]} = [Z].
\end{equation}
\end{thm}

\begin{remark}\label{bingo}
As in Theorem ~\ref{thmmain2} we could drop the assumption that $\rank
E_0=1$. 
Using the notation from above, we get
\begin{equation}\label{sprak}
    \sum_{k=p}^{N} \frac{1}{(2\pi i)^kk!} \tr(D\varphi_1 \cdots D\varphi_k
    R_{[k]}) = 
\sum_{k=p}^{N} \frac{1}{(2\pi i)^kk!} \tau D\varphi_1 \cdots D\varphi_k
    R_{[k]} \tau^{-1}=
 [Z],
\end{equation}
see Remark ~\ref{duo}. 
\end{remark}

It is natural to also consider the ``full'' currents
$D\varphi_1 \cdots D\varphi_k R_k$ and it would be interesting to
investigate whether they may capture geometric or algebraic
information (in addition to the fundamental cycle). In Section ~\ref{intressant} we compute the current
$D\varphi_1D\varphi_2 R^E_2$ for a Hermitian resolution of $Z$ from
Example ~\ref{midair}. We also illustrate Theorem ~\ref{thmnonpure} by
explicitly computing the currents in \eqref{jesse} in the situation of
Example ~\ref{condo}. 

\medskip 
\noindent 
\textbf{Acknowledgement:} We would like to thank Mats Andersson and
H\aa kan Samuelsson Kalm for valuable discussions on the topic of this
paper. 

\section{Preliminaries} \label{sectprel}

Throughout this paper $X$ will be a complex manifold of dimension $N$, and
$\chi(t)$ will be (a smooth approximand of) the characteristic
function of the interval $[1,\infty)$. 
Let $f$ be a holomorphic function on $X$ or, more generally, a
holomorphic section of a line bundle over $X$. Then there is an associated \emph{principal
  value current} $1/f$, \cites{Dol, HL}, 
defined, e.g., 
as the limit 
\begin{equation*}
\lim_{\epsilon\to 0}\chi(|f|^2/\epsilon)\frac{1}{f}. 
\end{equation*}
The associated \emph{residue current} is defined as $\dbar(1/f)$, cf.\
\eqref{martes}.

\subsection{Pseudomeromorphic currents} \label{ssectpm}

Following \cite{AW2} 
we say that a current of the form 
\begin{equation*}
      \frac{1}{z_{1}^{a_1}}\cdots\frac{1}{z_{k}^{a_k}}\dbar\frac{1}{z_{{k+1}}^{a_{k+1}}}\wedge\cdots\wedge\dbar\frac{1}{z_{m}^{a_m}}\wedge \xi,
\end{equation*}
where $z_1,\ldots, z_N$ is a local coordinate system and $\xi$ is a smooth form with compact support,
is an {\it elementary current}. 
Moreover a current on $X$ is said to be
\emph{pseudomeromorphic} if it can be written as a locally
finite sum of push-forwards of elementary currents under compositions
of modifications, open inclusions, or projections.\footnote{In
  \cite{AW2} only modifications were allowed. This more general class
  of pseudomeromorphic currents appeared in \cite{AS}.} 
Note that if $T$ is pseudomeromorphic, then so is $\dbar T$. 

The sheaf of pseudomeromorphic currents, denoted $\PM$, was introduced to
obtain a coherent approach to questions concerning principal value and
residue currents; in fact, all principal value and residue currents in
this paper 
are pseudomeromorphic.  It follows from, e.g., \cite{ALelong} that
currents of integration along analytic subvarieties $W\subset X$ are pseudomeromorphic.

In many ways pseudomeromorphic currents behave like normal
currents, i.e., currents $T$ such that $T$ and $dT$ are of order $0$. In particular, 
they satisfy the following \emph{dimension principle}, \cite{AW2}*{Corollary~2.4}.
\begin{prop} \label{proppmdim}
    If $T \in \PM(X)$ is a $(p,q)$-current with support on a
    subvariety $W\subset X$,
    and $\codim W > q$, then $T = 0$.
\end{prop}
Moreover, pseudomeromorphic currents admit natural restrictions to analytic subvarieties,
see \cite{AW2}*{Section~3} and also \cite{AS}*{Proposition~2.3}.
If $T \in \PM(X)$, $W \subset X$ is a subvariety of $X$, and $h$ is a tuple of
holomorphic functions such that $W = \{h=0\}$, the restriction
$\mathbf 1_W T$ can be defined, e.g., as 
\begin{equation*}
    {\bf 1}_W T := \lim_{\epsilon \to 0} \left(1-\chi(|h|^2/\epsilon)\right) T. 
\end{equation*}
This definition is independent of the choice of $\chi$ and the tuple $h$, and ${\bf 1}_W T$ is a pseudomeromorphic
current with support on $W$. 
If ${\bf 1}_WT=0$ for all subvarieties $W\subset X$ of positive codimension, then $T$ is
said to have the \emph{standard extension property, SEP}.

\subsection{Superstructure} 
Let  
\begin{equation*}
    0 \xrightarrow[]{} E_\nu \xrightarrow[]{\varphi_\nu} E_{\nu-1} \xrightarrow[]{\varphi_{\nu-1}}
    \cdots \xrightarrow[]{\varphi_2} E_1 \xrightarrow[]{\varphi_1} E_0 
\end{equation*}
be a complex of locally free $\Ok_X$-modules. 
Then $E := \bigoplus E_k$ has a natural superstructure, i.e., a $\mathbb{Z}_2$-grading,
which splits $E$ into odd and even elements $E^+$ and $E^-$, where $E^+ = \bigoplus E_{2k}$
and $E^- = \bigoplus E_{2k+1}$. Also $\End E$ gets a superstructure by letting the even elements be
the endomorphisms preserving the degree, and the odd elements the endomorphisms switching degrees.

We let $\E$ and $\E^\bullet$ denote the sheaves of smooth functions
and forms, respectively, on $X$ and we let $\E^\bullet(E) = \E^\bullet\otimes_{\E} \E(E)$ and
$\E^\bullet(\End E) = \E^\bullet\otimes_{\E} \E(\End E)$ 
be the sheaves of form-valued sections of $E$ and $\End E$,
respectively. 
Given a section $\gamma = \omega\otimes \eta$, where 
$\omega$ is a smooth form and $\eta$ is a smooth section of $E$ or $\End E$,
we let 
$\deg_f \gamma:=\deg \omega$ and
$\deg_e \gamma:=\deg \eta$.
Then $\E^\bullet(E)$ and $\mathcal{E}^\bullet(\End E)$ inherit superstructures by letting 
$\deg \gamma := \deg_f \gamma + \deg_e \gamma$. 
Both $\mathcal{E}^\bullet(E)$ and $\mathcal{E}^\bullet (\End E)$ are
naturally left $\E^\bullet$-modules. We make them into
right $\E^\bullet$-modules by letting 
\begin{equation}\label{motex}
\gamma \omega = (-1)^{(\deg \gamma)(\deg \omega)} \omega\gamma,
\end{equation}
where $\omega$ is a smooth form, and $\gamma$ is a section of
$\mathcal{E}^\bullet(E)$ or $\mathcal{E}^\bullet(\End E)$.
Moreover, if $\beta=\alpha\otimes\xi$ and $\gamma=\omega\otimes\eta, \gamma'=\omega'\otimes\eta'$ are sections of $\E^\bullet(E)$
and $\E^\bullet(\End E)$, respectively, we let 
\begin{eqnarray}
\notag\gamma(\beta)&=&(-1)^{(\deg_e\gamma)(\deg_f\beta)}
\omega\wedge\alpha \otimes\eta(\xi), \\
\label{circa}
\gamma \gamma'&=&(-1)^{(\deg_e\gamma)(\deg_f\gamma')}
\omega\wedge\omega' \otimes\eta \eta'. 
\end{eqnarray}
Note that if $\gamma=\alpha\otimes \Id$ then $\gamma \beta =
\alpha\beta$, $\gamma\gamma'=\alpha\gamma'$, and
$\gamma'\gamma=\gamma'\alpha$, 
cf.\ \eqref{motex}. Thus we can regard a form $\alpha$ as a
(form-valued) endomorphism. Moreover, we have the following
associativity: $(\gamma\gamma')
\beta=\gamma(\gamma'\beta)$ and $(\gamma
\gamma')\gamma''=\gamma(\gamma'\gamma'')$ if $\gamma''$ is a
section of $\E^\bullet (\End E)$. 
Analogously the sheaves $\mathcal C^\bullet(E) = \mathcal C^\bullet\otimes_{\E} \E(E)$ and
$\mathcal C^\bullet(\End E) = \mathcal C^\bullet\otimes_{\E} \E(\End
E)$ of current-valued sections of $E$ and $\End E$,
respectively, inherit superstructures. 

If $E_0,\dots,E_\nu$ (considered as vector bundles) are equipped with connections $D_{E_0},\dots,D_{E_\nu}$, and $D_E$ is the connection $\bigoplus D_{E_i}$
on $E$, we equip $\End E$ with the induced connection $D_{\End}$ defined by
\begin{equation} \label{eqDend}
    D_E(\gamma(\xi)) = D_{\End}(\gamma)\xi + (-1)^{\deg \gamma}
    \gamma (D_E \xi),
\end{equation}
where $\xi$ is a section of $\E^\bullet(E)$ and $\gamma$ is a section of $\E^\bullet(\End E)$.
It is then straightforward to verify that for arbitrary 
sections $\gamma, \gamma'$ of  $\E^\bullet(\End E)$,
\begin{equation} \label{eqleibniz}
    D_{\End}(\gamma\gamma') = D_{\End}\gamma\gamma' + (-1)^{\deg
      \gamma}\gamma D_{\End}\gamma'. 
\end{equation}
Moreover, note that if $\gamma=\alpha\otimes\Id$, then $D_{\End}
\gamma=d\alpha$, so, again, we can regard a form $\alpha$ as a
(form-valued) endomorphism. 

Throughout this paper we will use the sign conventions associated with
this superstructure, cf.\ Section ~\ref{matrisnot}.

\begin{ex} \label{exDkoszul}
    We consider the situation when $(E,\varphi)$ is the Koszul complex $(K,\phi) = (\bigwedge \Ok_X^{\oplus p}, \delta_f)$
    associated to a tuple $(f_1,\dots,f_p)$ of holomorphic functions, and $e_1,\dots,e_p$ is the standard basis of $\Ok_X^{\oplus p}$
    so that $\delta_f$ is contraction with $f = f_1 e_1^* + \dots + f_p e_p^*$, see Section~\ref{ssectkoszul}.
    If we assume that $D$ is trivial with respect to the induced bases $e_I$ of $(K,\phi)$, then $D\delta_f$ is contraction
    with $df_1\wedge  e_1^* + \dots + df_p \wedge e_p^*$. As $df_i \wedge e_i^*$ is even, we thus get that $D\phi_1 \cdots D\phi_p$ is
    contraction with
    \begin{equation*}
        (df_1 \wedge e_1^*+ \dots + df_p \wedge e_p^*)^p = p!~ df_1 \wedge e_1^* \wedge \dots \wedge df_p \wedge e_p^*
        = p! ~df_1 \wedge \dots \wedge df_p \wedge e_{\{1,\dots,p\}}^*,
    \end{equation*}
    where $e_{\{1,\dots,p\}}$ and $e_\emptyset$ are frames of $E_p$ and $E_0$, see Section~\ref{ssectkoszul} for notation.
    Thus
    \begin{equation} \label{eq:Dkoszul}
        D\phi_1 \cdots D\phi_p = p! ~df_1 \wedge \dots \wedge df_p
        \wedge e_\emptyset \wedge e_{\{1,\dots,p\}}^*. 
    \end{equation}
\end{ex}

\subsection{Residue currents associated with Hermitian locally free resolutions}\label{joo}

Let $\mathcal{G}$ be a coherent sheaf on $X$ of $\codim p > 0$  
with a Hermitian locally
free resolution $(E,\varphi)$, cf.\ the introduction. 
In \cite{AW1} Andersson and the second author defined a
($\Hom(E_0,E)$-valued) current $R^E$ associated with $(E,\varphi)$.
We will write $R^E = \sum R_k^E$, where $R^E_k$ is the part of $R^E$ which takes
values in $\Hom(E_0,E_k)$. The current $R^E_k$ is a $(0,k)$-current with support
on $\supp \mathcal{G}$ and thus  $R^E_k = 0$ if $k < p$ by the dimension principle,
Proposition ~\ref{proppmdim}.
The current $R^E$ satisfies  that if $\alpha$ is a holomorphic section of $E_0$, then 
$R^E \alpha = 0$ if and only if $\alpha$ belongs to $\im
\varphi_1$, \cite{AW1}*{Theorem~1.1}; this can be seen as a duality principle.
In particular, 
\begin{equation}\label{eqannright}
R^E_p \varphi_1 = 0. 
\end{equation}
The current $R^E$ is $\nabla$-closed, where $\nabla = \varphi - \dbar$, i.e.,
$\varphi_k R^E_k - \dbar R^E_{k-1}=0$ for all $k$. In particular,
\begin{equation} \label{eqannleft}
    \varphi_p R^E_p = 0.
\end{equation}

For details about the construction of these residue currents, we refer to \cite{AW1}.
For further reference, we mention that the construction is related to certain 
\emph{singularity subvarieties} associated to a coherent analytic sheaf, see \cite{ST}.
The singularity subvariety $Z^E_k$ is defined as the set where $\varphi_k$ does not have optimal
rank\footnote{In \cite{ST}, the singularity subvariety $S_\ell(\mathcal{G})$ is defined
as the set of $x$ such that $\depth_x \mathcal{G}_x \leq \ell$, and by the Auslander-Buchsbaum formula,
$Z^E_k = S_{N-k}(\mathcal{G})$.}.
By uniqueness of minimal free resolutions, these sets are in fact independent of the
choice of $(E,\varphi)$, and indeed only depend on $\mathcal{G}$.

\begin{ex}
Assume that $Z$ is a complete intersection of codimension $p$,
i.e., $\mathcal O_Z=\mathcal O_X/\mathcal J$, where $\mathcal J$ is a complete
intersection ideal, generated by, say, $f = (f_1,\dots,f_p)$.
Let $(E,\varphi)$ be the Koszul complex of $f$.
Then the corresponding sheaf complex is a free resolution of $\Ok_Z$ and 
\begin{equation} \label{eq:bmch}
    R^E_p = \dbar\frac{1}{f_p} \wedge \cdots \wedge
    \dbar\frac{1}{f_1}\wedge e_{\{1,\dots,p\}} \wedge e_\emptyset^*, 
\end{equation}
where $e_{\{1,\dots,p\}}$ and $e_\emptyset$ are frames of $E_p$ and $E_0$, see Section~\ref{ssectkoszul}
for notation. This was proven in  \cite{PTY}*{Theorem~4.1}\footnote{
In fact, in \cite{PTY} it was
  proved that $\dbar(1/f_p)\wedge\cdots\wedge\dbar (1/f_1)$ equals the
  so-called Bochner-Martinelli residue current 
  of $f$, which by \cite{AndCH}*{Corollary~3.5} is the coefficient of $R^E$ (i.e., the
  current in front of $e_{\{1,\ldots, p\}}\wedge e_\emptyset^*$).} and
  \cite{AndCH}*{Corollary~3.5}. 
\end{ex}

\subsection{A comparison formula for residue currents} \label{ssectcomp}

Let $\alpha : H \to G$ be a homomorphism of finitely generated $\Ok_{X,\zeta}$-modules,
and let $(F,\psi)$ and $(E,\varphi)$ be free resolutions of $H$ and $G$ respectively.
We say that a morphism of complexes $a : (F,\psi) \to (E,\varphi)$ \emph{extends} $\alpha$ if the map
$\coker \psi_1 \cong H \to G \cong \coker \varphi_1$
induced by $a_0$ equals $\alpha$.

\begin{prop} \label{propcomplexcomparison}
    Let $\alpha : H \to G$ be a homomorphism of finitely generated $\Ok_{X,\zeta}$-modules, and let $(F,\psi)$ and
    $(E,\varphi)$ be free resolutions of $H$ and $G$ respectively.
    Then, there exists a morphism $a : (F,\psi) \to (E,\varphi)$ of complexes which extends $\alpha$.

    If $\tilde{a} : (F,\psi) \to (E,\varphi)$ is any other such morphism, 
    then there exists a morphism $s_0 : F_0 \to E_1$ such that $a_0 - \tilde{a}_0 = \varphi_1 s_0$.
\end{prop}

The existence of $a$ follows from defining it inductively by a relatively
straightforward diagram chase, see \cite{Eis}*{Proposition~A3.13},
and the existence of $s_0$ follows by a similar argument.

The residue currents associated with $(E, \varphi)$
and $(F, \psi)$ are
related by the following comparison formula, see \cite{LComp}*{Theorem~3.2}. 

\begin{thm} \label{thmRcomparison}
Assume that $H$ and $G$ are two finitely generated $\Ok_{X,\zeta}$-modules 
with Hermitian free resolutions
    $(F,\psi)$ and $(E,\varphi)$, respectively.
    If $a : (F,\psi) \to (E,\varphi)$ is a morphism of complexes, then
    \begin{equation} \label{eq:comparisongeneral}
R^Ea_0 - aR^F = \nabla M
    \end{equation}
where $M$ is a pseudomeromorphic $\Hom(F_0,E)$-valued current with support on $\supp H \cup \supp G$.
 \end{thm}

If we write $M = \sum M_\ell$, where $M_\ell$ is the part of $M$ with values
in $\Hom(F_0,E_\ell)$, and if $H$ and $G$ have codimension $\geq k$,
then $M_k = 0$ by \cite{LComp}*{Corollary ~3.6}. 
In particular, if $H$ and $G$ have codimension $p$, then \eqref{eq:comparisongeneral}
implies that (by taking the $\Hom(F_0, E_p)$-valued part) 
\begin{equation} \label{eq:comparisoncodimp}
R^E_p a_0 = a_p R^F_p+\varphi_{p+1} M_{p+1}.  
\end{equation}
If, in addition, $G$ is Cohen-Macaulay, i.e., it has a free resolution
of length $p$, and $(E, \varphi)$ is such a 
resolution, then  
\begin{equation} \label{eq:comparisoncm}
R^E_p a_0 = a_p R^F_p.
\end{equation}
Finally, we will also need to consider the situation when $G$ is Cohen-Macaulay,
but when the free resolution does not have minimal length $p$. The following
Lemma follows from \cite{LComp}*{Lemma~3.3 and Corollary~3.6}. 

\begin{lma} \label{lmampplus1}
    We use the notation from Theorem~\ref{thmRcomparison}. Assume that
    $H$ and $G$ have codimension $p$ and moreover that $G$ is Cohen-Macaulay.
    Then
    \begin{equation*}
        M_{p+1} = -\sigma_{p+1}^E a_p R^F_p, 
    \end{equation*}
where $\sigma_{p+1}^E$ is smooth. 
\end{lma}

\subsection{Matrix notation}\label{matrisnot}

For a section $\gamma$ of $\E^\bullet(\End E)$ (or $\mathcal{C}^\bullet(\End E)$), let $\{ \gamma \}$
denote the matrix representing $\gamma$ in a local frame of $E$.

From \eqref{circa} it follows that if $\beta$ and $\gamma$ are sections of $\mathcal{E}^\bullet(\End E)$, then
\begin{equation} \label{eqmatrixmultsign}
    \{\beta \gamma \} = (-1)^{(\deg_e \beta)(\deg_f \gamma)} \{\beta\}\{\gamma\}.
\end{equation}
If we consider the main formula \eqref{eqmain} as a product of matrices in a local frame, then by
repeatedly using \eqref{eqmatrixmultsign}, the formula becomes
\begin{equation*}
    [Z] = \frac{1}{(2\pi i)^p p!} (-1)^{p(p-1)/2+p^2} \{D\varphi_1\}\cdots \{D\varphi_p\}\{R^E_p\}.
\end{equation*}
In \cite{LW1}, we explicitly computed the current $D\varphi_1 \cdots D\varphi_p R^E_p$, 
when $(E,\varphi)$ is a certain free resolution of a 2-dimensional Artinian monomial ideal,
by multiplying matrices, and this is the reason for why the constant $C_p = (-1)^{p(p-1)/2+p^2}=(-1)^{\lceil p/2 \rceil}$ appeared in
\cite{LW1}*{(7.4)}. 

When $(E,\varphi)$ is the Koszul complex of a tuple $(f_1,\dots,f_p)$ of holomorphic
functions defining a complete intersection ideal $\mathcal J(f)$ of codimension $p$, then
\begin{equation} \label{eqplcomputation}
\begin{gathered}
\frac{1}{(2\pi i)^p p!}\{ D\varphi_1 \cdots D\varphi_p R^E_p\} =
    \frac{1}{(2\pi i)^p p!} (-1)^{p^2} \{ D\varphi_1 \cdots D\varphi_p\}\{ R^E_p\} = \\
=
    \frac{1}{(2\pi i)^p } (-1)^{p^2} df_1 \wedge \dots \wedge df_p \wedge 
    \dbar \frac{1}{f_p} \wedge \dots \wedge \dbar\frac{1}{f_1} = \\
=
    \frac{1}{(2\pi i)^p }
    \dbar \frac{1}{f_p} \wedge \dots \wedge \dbar\frac{1}{f_1} \wedge
    df_1 \wedge \dots \wedge df_p = [Z],
\end{gathered} 
\end{equation}
where $\Ok_Z=\Ok_X/\mathcal J(f)$ and 
where we have used \eqref{eq:Dkoszul} and \eqref{eq:bmch} in the
second equality 
and the Poincar\'e-Lelong formula, \eqref{eqpl}, in the last equality.

Assume that $\beta$ and $\gamma$ are $\Hom(E_k,E_\ell)$- and
$\Hom(E_\ell, E_k)$-valued forms, respectively. 
Using that for scalar-valued $(i\times j)$- and $(j\times i)$-matrices $B$ and $C$, $\tr (BC) = \tr (CB)$,
together with \eqref{eqmatrixmultsign}, one gets that
\begin{align*} 
    \tr \{ \beta\gamma \} &= (-1)^{(\deg_e \beta) (\deg_f \gamma) } \tr \big(\{\beta\} \{\gamma\} \big)= \\
    &= (-1)^{(\deg_e \beta) (\deg_f \gamma) + (\deg_f \beta)(\deg_f \gamma)} \tr \big(\{\gamma\} \{\beta\} \big)= \\
    &= (-1)^{(\deg_e \beta) (\deg_f \gamma) + (\deg_f \beta)(\deg_f \gamma) + (\deg_e \gamma)(\deg_f \beta)} \tr \{\gamma\beta\}.
\end{align*}
Hence,
\begin{equation}
    \label{eq:trABBA}
    \tr \big(\beta\gamma \big )= (-1)^{(\deg\beta)(\deg \gamma) - (\deg_e \beta)(\deg_e \gamma)} \tr \big (\gamma\beta\big).
\end{equation}
Note that both \eqref{eqmatrixmultsign} and \eqref{eq:trABBA} hold
also when either $\beta$ or $\gamma$ is a section of $\mathcal
C^\bullet (\End E)$.

\section{Universal free resolutions} \label{sectkoszul}

A key ingredient in the proof of Theorem~\ref{thmmain2} is a specific universal free
resolution of $\Ok_{Z,\zeta}$ for $\zeta$ where $Z$ is Cohen-Macaulay.
It is in general far from minimal, but on
the other hand the construction is explicit.  
The universal free resolution, which is a Koszul complex over a certain ring $A$ that we describe
below, is a special case of a universal free resolution of
Cohen-Macaulay ideals due to Scheja and Storch, \cite{SS}*{p. 87--88}, and
Eisenbud, Riemenschneider and Schreyer, \cite{ERS}*{Theorem~1.1 and Example~1.1},
who however do this in an algebraic setting.

In order to prove Theorem~\ref{thmmain2}, it will be enough to have a free resolution generically on $Z$.
Generically on $Z$, a Noether normalization $\pi : Z \to W$ is given by a projection to $W := Z_{\rm red}$,
and one can there describe $\Ok_{Z,\zeta}$ as a free $\Ok_{W,\zeta}$-module
in an explicit way, see Lemma~\ref{lmagoodcoord2}.
In Lemma~\ref{lma:mainKoszul}, which we use to prove Theorem~\ref{thmmain2},
we will use this description of $\Ok_{Z,\zeta}$ as a free $\Ok_{W,\zeta}$-module.
In this case, we can give a direct proof that the construction of \cite{ERS}
and \cite{SS} indeed gives a free resolution of $\Ok_Z$; this is Theorem~\ref{thmkoszul}.

\subsection{The ring $A$} \label{ssect:ringA}

For a tuple $\alpha = (\alpha_1,\dots,\alpha_p) \in \Np^p$, where $\Np = \{0,1,2\dots\}$, we use
the multi-index notation $z^\alpha := z_1^{\alpha_1}\cdots z_p^{\alpha_p}$,
and, in addition, we let $|\alpha| := \alpha_1 + \cdots + \alpha_p$.

\begin{lma} \label{lmagoodcoord2}
    Let $X$ be a complex manifold of dimension $N$, and assume that $\mathcal{J} \subset \Ok_X$ is the defining ideal
    of an analytic subspace $Z$ of $X$, of pure codimension $p$, and let $n = N-p$.
    Let $W = Z_{\rm red}$ and assume that  $\xi \in W_{\rm reg}$. 
    Assume that we near $\xi$ have coordinates $(z,w) \in \C^p \times \C^n$ on $X$,
    such that $(z,w)(\xi) = 0$ and that in these coordinates, $W = \{ z_1 = \dots = z_p = 0 \}$.
    Let $m$ denote the geometric multiplicity of $W$ in $Z$ near $\xi$.

    Then there exist a neighbourhood $U \subset W$ of $\xi$, a hypersurface $Y \subset U$, and tuples
    $\alpha^1,\dots,\alpha^m \in \Np^p$ such that for $\zeta \in U\setminus Y$,
    $\Ok_{Z,\zeta}$ is a free $\Ok_{W,\zeta}$-module
    with a basis $z^{\alpha^1},\dots,z^{\alpha^m}$.
    Moreover, the tuples $\alpha^i$ satisfy $|\alpha^1| \geq
    |\alpha^2| \geq \cdots \geq |\alpha^m|$ 
    and if we express any monomial 
    $z^\gamma$ in terms of the $z^{\alpha^i}$,
    $z^\gamma = \sum f_i(w) z^{\alpha^i} + \mathcal{J}$, then for
    all $i$ such that $f_i \not\equiv 0$, we have that $|\alpha^i| \geq |\gamma|$.
\end{lma}

Note that if one considers a tuple $\beta \in \Np^p$, then, by the last statement
of the lemma, we have for each $j$,
\begin{equation} \label{eqzbetaalphaj}
    z^\beta z^{\alpha^j} = \sum_{i \leq j} f_i(w) z^{\alpha^i} + \mathcal{J},
\end{equation}
and if $\beta \neq 0$, then the sum can be taken just over $i < j$.

\begin{proof}
    By the Nullstellensatz in $\Ok_{X,\xi}$, we can choose $\beta_i$ such that
    $z_i^{\beta_i} \in \mathcal{J}$ for $i = 1,\dots,p$.
    In particular, the finite set of monomials $z^{\alpha}$ such that
    $\alpha_i < \beta_i$ for $i=1,\dots,p$ must generate $\Ok_{Z,\xi}$ as an $\Ok_{W,\xi}$-module.
    By coherence, these monomials also generate $\Ok_{Z,\zeta}$ as an $\Ok_{W,\zeta}$-module
    for $\zeta$ in some neighbourhood $U \subset W$ of $\xi$.

We let $a^i$ be an enumeration of the tuples $\alpha$ with $\alpha_k < \beta_k$ for $k=1,\dots,p$,
ordered so that $|a^i| \geq |a^j|$ if $i \leq j$.
We now choose $\alpha^1,\dots,\alpha^M$ inductively among the $a^i$
so that $z^{\alpha^1},\dots,z^{\alpha^M}$ are independent over $\Ok_{W,\xi}$
in the following way:
First, we let $\alpha^1 = a^{i_1}$, where $i_1$ is the first index $i$ such
that $f_1(w) z^{a^i} \equiv 0$ in $\Ok_{Z,\xi}$ implies that
$f_1 \equiv 0$. Then, if we have already chosen $\alpha^1,\dots,\alpha^k$,
$\alpha^j = a^{i_j}$, we define inductively $\alpha^{k+1} = a^{i_{k+1}}$ as the next
$a^i$ such that if
$f_1(w) z^{\alpha^1} + \cdots + f_k(w) z^{\alpha^k} + f_{k+1}(w) z^{a^i} \equiv 0$,
then $f_{k+1} \equiv 0$.
Clearly, $|\alpha^1| \geq \cdots \geq |\alpha^M|$.

Note that if $a^k$ is not among the $\alpha^i$, then there exists a relation
$f_k(w) z^{a^k} = \sum_{j: i_j < k} g_{k,j}(w) z^{\alpha^j}$ in $\Ok_{Z,\xi}$,
where $f_k \not\equiv 0$.
By possibly shrinking $U$, we can assume that all the $f_k$'s are defined on $U$.
Let $Y := \bigcup_{k \notin \{i_1,\dots,i_M\}} \{ f_k = 0 \}$. Then,
outside the hypersurface $Y$, any such $z^{a^k}$ can be expressed
uniquely in terms of $z^{\alpha^j}$ with $i_j < k$.
Thus, for $\zeta \in U \setminus Y$, $\Ok_{Z,\zeta}$ is a free $\Ok_{W,\zeta}$-module
with basis $z^{\alpha^1},\dots,z^{\alpha^M}$. Therefore $M=m$. 
In addition, since each $z^{a^k}$ not among the $z^{\alpha^i}$ can be written
in terms of $z^{\alpha^j}$, with $i_j < k$, by the ordering of the $a^i$,
those $\alpha^i$ will satisfy that $|\alpha^i|\geq |a^k|$.
\end{proof}

\begin{df} \label{def:A-module}
    We consider the situation in Lemma~\ref{lmagoodcoord2}. Given $\zeta \in U \setminus Y$,
    we define the $\Ok_{W,\zeta}$-module
    \begin{equation*}
        A  = A_\zeta := \Ok_{X,\zeta} \otimes_{\Ok_{W,\zeta}} \Ok_{Z,\zeta}.
    \end{equation*}
\end{df}

Note that by Lemma~\ref{lmagoodcoord2}, $\Ok_{Z,\zeta}$ is a
free $\Ok_{W,\zeta}$-module of rank $m$, so $A$ is a free $\Ok_{X,\zeta}$-module
of rank $m$, i.e., $A \cong \Ok_{X,\zeta}^{\oplus m}$.
We will denote an element $f \otimes g \in A$ by $f \cl{g}$.
We will also sometimes use the short-hand notation $f := f[1]$ and $[g] := 1[g]$.
Note that since $\Ok_{X,\zeta}$ and $\Ok_{Z,\zeta}$ are $\Ok_{W,\zeta}$-algebras, so is $A$,
and the multiplication is defined by $(f_1\cl{g_1})(f_2\cl{g_2}) = f_1f_2\cl{g_1g_2}$.

\begin{remark} \label{rem:upper-triangular}
    Using the notation from above, for $\zeta \in U\setminus Y$, we have
    a basis $z^{\alpha^1},\dots,z^{\alpha^m}$ of $\Ok_{Z,\zeta}$ as a free $\Ok_{W,\zeta}$-module.
    This gives a basis $\cl{z^{\alpha^1}},\dots,\cl{z^{\alpha^m}}$ of $A$ as a free $\Ok_{X,\zeta}$-module.
    If $z^\gamma$ is a monomial, then we can consider (multiplication
    with) $\cl{z^\gamma}$ as an element in $\End_{\Ok_{X,\zeta}}(A)$,
    and the matrix of $\cl{z^\gamma}$ with respect to the basis $\cl{z^{\alpha^1}},\dots,\cl{z^{\alpha^m}}$
    from Lemma~\ref{lmagoodcoord2} is upper triangular by \eqref{eqzbetaalphaj}, and it has zeros
    along the diagonal unless $\gamma = 0$, in which case $\cl{z^\gamma}$ is the identity matrix.
\end{remark}

\subsection{Universal free resolutions} \label{ssectkoszul}

Let $R$ be a commutative ring, and let $x_1,\dots,x_p$ be elements of $R$.
To fix notation, we remind that the \emph{Koszul complex} of $x = (x_1,\dots,x_p)$
is the complex $(\bigwedge^\bullet R^{\oplus p},\delta_x)$, where the differential $\delta_x$ is defined
by inner multiplication with $x$, i.e., if we choose as a standard basis $e_1,\dots,e_p$
of $R^{\oplus p}$, then
\begin{equation*}%\label{eqkoszul}
    \delta_x : e_I \mapsto \sum_{i=1}^k (-1)^{i-1} x_{I_i} e_{I\setminus I_i},
\end{equation*}
where $I = (I_1,\dots,I_k)$, and we use the short-hand notation
$e_I = e_{I_1}\wedge\cdots\wedge e_{I_k}$.
In particular, we use the notation $e_\emptyset$ for the basis of $\bigwedge^0 R^{\oplus p} \cong R$.
If the sequence $x$ is a regular sequence, then it is well-known that $(\bigwedge^\bullet R^{\oplus p},\delta_x)$ is
a free resolution of $R/(x_1,\dots,x_p)$, see for example \cite{Eis}*{Corollary~17.5}.
When $R = \Ok_{X,\zeta}$, then $f = (f_1,\dots,f_p)$ is a regular sequence 
if and only if $\codim \{ f_1 = \cdots = f_p = 0 \} = p$. 
Hence, for complete intersection ideals, we have an explicitly defined free resolution.
The universal free resolution gives an explicit free resolution for more
general ideals in $\Ok_{X,\zeta}$, but then one considers a Koszul complex over the ring $A$ instead of over $\Ok_{X,\zeta}$.

\begin{thm} \label{thmkoszul}
    Assume that we are in the situation of Lemma~\ref{lmagoodcoord2}, and that we fix
    some $\zeta \in U \setminus Y$. Let $A$ be as in Definition~\ref{def:A-module},
    and let $\mathbf{z}_i := z_i - \cl{z_i} \in A$ for $i = 1,\dots,p$.
    Then, the Koszul complex $(K,\phi) := (\bigwedge^\bullet A^{\oplus p},\delta_{\mathbf{z}})$ of
    $\mathbf{z} := (\mathbf{z}_1,\dots,\mathbf{z}_p)$ is a free
    resolution of $\Ok_{Z,\zeta}$ over $A$ and $\Ok_{X,\zeta}$. 
%    as $A$-modules and as $\Ok_{X,\zeta}$-modules.
\end{thm}

For tuples $\gamma,\eta \in \Np^p$ we use the partial ordering
that $\gamma \leq \eta$ if and only if $\gamma_i \leq \eta_i$ for $i=1,\dots,p$.
We also use the short-hand notation ${\bf 1} = (1,\dots,1) \in \Np^p$.

For the convenience of the reader, we provide a proof of Theorem~\ref{thmkoszul} in our situation.
In \cite{SS} and \cite{ERS}, the corresponding theorem is in an algebraic setting, and does not
immediately apply to our setting, although it should be possible to adapt the proof to our setting
using analytic tensor products (cf., Section~2 in \cite{ABM}).

\begin{proof}
    By construction, $K$ consists of free $A$-modules, and since, as explained above,
    $A$ is a free $\Ok_{X,\zeta}$-module, $K$ is also a complex of free $\Ok_{X,\zeta}$-modules.
    Exactness is independent of whether we consider the complex as $\Ok_{X,\zeta}$-modules
    or $A$-modules, so it is sufficient to prove that $(K,\phi)$
    is a free resolution as $\Ok_{X,\zeta}$-modules.

    We first prove that $\coker \phi_1 \cong \Ok_{Z,\zeta}$.
    We get a surjective mapping $\pi : K_0 \to \Ok_{Z,\zeta}$ by letting
    $\pi(f\cl{g}) := fg$. Note that $\pi(\mathbf{z}_i) = 0$ for $i = 1,\dots,p$,
    so we get a well-defined induced mapping $\tilde{\pi} : K_0/(\im \phi_1) \to \Ok_{Z,\zeta}$.
    Clearly, $\tilde{\pi}$ is surjective since $\pi$ is surjective.
    Next, we claim that
    \begin{equation} \label{eqfg}
        f [g] = [fg] + \sum \mathbf{z}_i \eta_i,
    \end{equation}
    for some $\eta_i \in A$, $i=1,\dots,p$.
    To prove the claim, we first choose $\beta_i$ such that $z_i^{\beta_i} = 0$ in $\Ok_{Z,\zeta}$ for $i=1,\dots,p$, which 
    is possible by the Nullstellensatz.
    We then make a finite Taylor expansion of $f$,
    \begin{equation*}
        f = \sum_{\alpha \leq \beta-{\mathbf 1}} f_\alpha(w) z^\alpha + \sum_{i=1}^p z_i^{\beta_i} f_i(z,w).
    \end{equation*}
    Using this Taylor expansion, in combination with the formula
    \begin{equation*}
        z_i^k = \cl{z_i^k} + (z_i^{k-1} + z_i^{k-2} \cl{z_i} + \dots + \cl{z_i^{k-1}})(z_i - \cl{z_i}),
    \end{equation*}
    and the fact that $\cl{z_i^{\beta_i}} = 0$ and $f_\alpha(w)\cl{g} = \cl{f_\alpha(w) g}$,
    we get that $f[g]$ is of the form \eqref{eqfg}.
    If $\tilde{\pi}(\sum f_i \cl{g_i} e_\emptyset) = 0$, then $\sum f_i g_i = 0$ in $\Ok_{Z,\zeta}$,
    and by \eqref{eqfg}, 
%$\sum f_i \cl{g_i} e_\emptyset = \phi_1 \eta$,
\[
\sum f_i \cl{g_i} e_\emptyset = \sum\mathbf z_j \eta_j e_\emptyset = \phi_1
\eta, 
\]
    for some $\eta =(\eta_1,\ldots, \eta_p)\in K_1$,
    i.e., $\sum f_i \cl{g_i} = 0$ in $\coker \phi_1$, so
    $\tilde{\pi}$ is injective. We thus get that $\coker \phi_1 \cong \Ok_{Z,\zeta}$.

    It remains to see that $(K,\phi)$ is exact at levels $k \geq 1$.
    In order to prove this, we first prove that $\phi_1$ is pointwise surjective outside of
    $W = \{ z_1 = \dots = z_p = 0 \}$. If $(z,w) \notin W$, we can assume that, say, $z_i \neq 0$. Then
    $\mathbf{z}_i$ is invertible, with inverse
    \begin{equation*}
        \gamma_i := \sum_{k=0}^\infty \frac{1}{z_i^{k+1}}\cl{z_i^{k}},
    \end{equation*}
    where the series is in fact a finite sum, since $z_i^k = 0$ in $\Ok_{Z,\zeta}$
    for $k \gg 1$ by the Nullstellensatz.
    Then, $\phi_1(f[g]\gamma_i e_i) = f[g] e_\emptyset$, so $\phi_1$ is surjective as a morphism of sheaves.
    Since the image is $K_0$, which is a vector bundle, it is also pointwise surjective.
    To conclude, $\phi_1$ is pointwise surjective outside of $W$, i.e., $Z^K_1 \subset W$,
    where $Z^K_1$ is the first singularity subvariety associated to $(K,\phi)$.

    We next prove that the complex is exact as a complex of sheaves at level $k \geq 1$ outside of $W$.
    As above, if $(z,w)$ is outside of $W$, and, say, $z_i \neq 0$, and $\alpha \in K_k$ is such that
    $\phi_k  \alpha = 0$, then $\phi_{k+1} (\gamma_i e_i \wedge \alpha) = (\delta_{\mathbf{z}} \gamma_i e_i)\wedge \alpha = \alpha$,
    so the complex is exact as a complex of sheaves outside of $W$. For a free resolution $(E,\varphi)$,
    $Z_{k+1}^E \subset Z_k^E$, for $k \geq 1$, see \cite{Eis}*{Corollary~20.12}.
    Hence, $Z_k^K\setminus W \subset Z_1^K\setminus W = \emptyset$, i.e., $Z_k^K \subset W$ for $k \geq 1$.

    To conclude, the complex $(K,\phi)$ of length $p$ is pointwise exact outside of $W$,
    which has codimension $p$.
    Thus, it is exact as a complex of sheaves by the Buchsbaum-Eisenbud criterion,
    \cite{Eis}*{Theorem~20.9}, because $\codim Z_k^K \geq p \geq k$ and 
    the pointwise exactness of $(K,\phi)$ outside of $W$ implies that
    $\rank K_k = \rank \phi_k + \rank \phi_{k+1}$.
\end{proof}

In general, for $\zeta \in U\setminus Y$, the universal free resolution of $\Ok_{Z,\zeta}$ is not minimal as a
free resolution of $\Ok_{X,\zeta}$-modules.
To see this, note that $K_0 \cong A$, so if $Z$ has geometric multiplicity $m > 1$ near $\zeta$,
then $\rank_{\Ok_{X,\zeta}} K_0 = m > 1$, while a minimal free resolution $(E,\varphi)$ of $\Ok_{Z,\zeta}$
would have $\rank_{\Ok_{X,\zeta}} E_0 = 1$.

\section{Proofs of Theorem~\ref{thmmain2} and Theorem~\ref{thmnonpure}} \label{sectmainproof}

A key part in the proof of Theorem~\ref{thmmain2}
is to prove that the currents on the left-hand sides of \eqref{eqmain2}
and \eqref{eqmain3} are independent of the choice of locally free resolution
$(E,\varphi)$ of $\Ok_Z$ and the choice of connections on $E_0,\dots,E_p$.

\begin{lma} \label{lmaindep}
    Let $G$ be a finitely generated $\Ok_{X,\zeta}$-module of codimension $p$,
    and let $(E,\varphi)$ and $(F,\psi)$ be Hermitian free resolutions of $G$.
    Then,
    \begin{equation} \label{eqindep}
        \tr (D\varphi_1 \cdots D\varphi_p R^E_p) = \tr (D\psi_1 \cdots D\psi_p R^F_p),
    \end{equation}
    where $D$ is the connection on $\End(E\oplus F)$ induced by arbitrary
    connections on $E_0,\dots,E_p$ and $F_0,\dots,F_p$.
\end{lma}

\begin{lma} \label{lmaindep2}
    Let $G$, $(E,\varphi)$, $(F,\psi)$ and $D$ be as in Lemma~\ref{lmaindep},
    and let $\eta$ and $\tau$ be the natural surjections
    $\eta : F_0 \to \coker \psi_1 \cong G$ and $\tau : E_0 \to \coker \varphi_1 \cong G$.
Then     
    \begin{equation} \label{eqindep2}
        \eta D\psi_1 \cdots D\psi_p R^F_p \eta^{-1} = \tau D\varphi_1 \cdots D\varphi_p R^E_p \tau^{-1}.
    \end{equation}
\end{lma}
Here $R^F_p \eta^{-1}$ and  $R^E_p \tau^{-1}$ are defined as in the
text preceding Theorem ~\ref{thmmain2}.  

\begin{remark} \label{rem:DP}
In case $\rank E_0 = \rank F_0 = 1$ these lemmas coincide.
When $G = \Ok_{X,\zeta}/\mathcal{I}$, where $\mathcal{I}$ is a complete intersection ideal of
codimension $p$, and $(E,\varphi)$ and $(F,\psi)$ are Koszul complexes
of two minimal sets of generators 
of $\mathcal{I}$, then \eqref{eqindep} and \eqref{eqindep2} follows
rather easily from the transformation law and duality principle 
for Coleff-Herrera products. 
This was a key observation which allowed for global versions
of the Poincar\'e-Lelong formula \eqref{eqpl} for locally complete intersections in \cite{DP}.
In order to prove Lemma~\ref{lmaindep} and Lemma~\ref{lmaindep2}, we use the comparison formula,
Theorem~\ref{thmRcomparison}, which is a generalization of the
transformation law.
\end{remark}

The proofs of both of these lemmas use the following lemma.

\begin{lma} \label{lma:DpsiDphi}
    Let $(E,\varphi)$ and $(F,\psi)$ be complexes of free $\Ok_{X,\zeta}$-modules, and let
    $b : (E,\varphi) \to (F,\psi)$ be a morphism of complexes. Let $D$ be the connection 
    on $\End(E\oplus F)$ induced by connections on $E_0,\dots,E_p$ and $F_0,\dots,F_p$.
    Then,
    \begin{equation} \label{eq:DpsiDphi}
        D\psi_1 \cdots D\psi_p b_p = b_0 D\varphi_1 \dots D\varphi_p 
        + \psi_1 \alpha + \beta \varphi_p
    \end{equation}
    for a smooth $\Hom(E_p,F_1)$-valued $(p,0)$-form $\alpha$ and a smooth
    $\Hom(E_{p-1},F_0)$-valued $(p,0)$-form $\beta$.
\end{lma}

\begin{proof}
    We claim that for any $1 \leq k \leq p$,
    \begin{equation} \label{eqmoveb}
    \begin{gathered}
       D\psi_1 \cdots D\psi_k b_k D\varphi_{k+1} \cdots D\varphi_p = \\
       D\psi_1 \cdots D\psi_{k-1} b_{k-1} D\varphi_{k} \cdots D\varphi_p
        + \psi_1 \alpha_k + \beta_k \varphi_p
    \end{gathered}
    \end{equation}
    for a smooth $\Hom(E_p,F_1)$-valued $(p,0)$-form $\alpha_k$ and a smooth
    $\Hom(E_{p-1},F_0)$-valued $(p,0)$-form $\beta_k$.
    By using this repeatedly for $k=p,\dots,1$, we get \eqref{eq:DpsiDphi}.

   To prove the claim, we note first that since $b$ is a morphism of complexes, $\psi_k b_k = b_{k-1} \varphi_k$,
   and thus,
   \begin{equation} \label{eqdgpap}
   \begin{gathered}
       D\psi_k b_k = D(\psi_k b_k) + \psi_k Db_k = D(b_{k-1} \varphi_k) + \psi_k Db_k = \\
        Db_{k-1} \varphi_k + b_{k-1}D\varphi_k + \psi_k Db_k,
   \end{gathered}
   \end{equation}
   where the signs depend on that $b$ is an even mapping, while $\psi$ is odd.

   We now replace $D\psi_k b_k$ in the first line of \eqref{eqmoveb} by the expression in the
   second line of \eqref{eqdgpap}.
   Note first that the term coming from the term $b_{k-1} D\varphi_k$ in the second
   line of \eqref{eqdgpap} equals the first term of the second line of \eqref{eqmoveb}.

   We consider next the term
   \begin{equation} \label{eqdbphi}
       D\psi_1 \cdots D\psi_{k-1}Db_{k-1} \varphi_k D\varphi_{k+1} \cdots D\varphi_p,
   \end{equation}
   coming from the term $Db_{k-1}\varphi_k$ in the second line of \eqref{eqdgpap}.
   Since $\varphi_\ell \varphi_{\ell+1} = 0$, we get by the Leibniz rule \eqref{eqleibniz}
   and the fact that $\varphi_\ell$ has odd degree that
   $\varphi_\ell D\varphi_{\ell+1} = D\varphi_\ell \varphi_{\ell+1}$.
   Using this repeatedly for $\ell=k,\dots,p-1$, we get that \eqref{eqdbphi} equals
   \begin{equation*}
       D\psi_1 \cdots D\psi_{k-1}Db_{k-1} D\varphi_k D\varphi_{k+1} \cdots D\varphi_{p-1} \varphi_p
       =: \beta_k \varphi_p.
   \end{equation*}

   Finally, we consider the term coming from the term $\psi_k Db_k$ in
   the second line of \eqref{eqdgpap}. 
   By using that $\psi_\ell \psi_{\ell+1} = 0$ and the Leibniz rule, we get that
   $D\psi_\ell \psi_{\ell+1} = \psi_\ell D\psi_{\ell+1}$, and using
   this repeatedly we get that this term equals
   \begin{equation*}
       \psi_1 D\psi_2 \cdots D\psi_k Db_k D\varphi_{k+1}\cdots D\varphi_p =: \psi_1 \alpha_k.
   \end{equation*}

   To conclude, when replacing $D\psi_k b_k$ in the first line of \eqref{eqmoveb}
   by the last line of \eqref{eqdgpap}, we obtain three terms
   of the form as in the second line of \eqref{eqmoveb},
   and we have thus proved \eqref{eqmoveb}.
\end{proof}

\begin{proof}[Proof of Lemma~\ref{lmaindep}]
    Since $G$ has codimension $p$, it is Cohen-Macaulay outside of a subvariety of codimension $p+1$. Since
    both sides of \eqref{eqindep} are pseudomeromorphic $(p,p)$-currents, it is by the dimension
    principle, Proposition~\ref{proppmdim}, enough to prove \eqref{eqindep} where $G$ is Cohen-Macaulay.
    We will thus assume for the remainder of the proof that $G$ is Cohen-Macaulay.

    Let $(H,\eta)$ by any free resolution of $G$.
    Using \eqref{eq:trABBA} and \eqref{eqannright},
    we get that if $\xi : H_p \to H_1$ is any smooth morphism, then
    \begin{equation} \label{eqannR}
        \tr (\eta_1 \xi R^H_p) = \pm \tr (\xi R^H_p \eta_1) = 0.
    \end{equation}

    We let $a : (F,\psi) \to (E,\varphi)$ and $b : (E,\varphi) \to (F,\psi)$ be morphisms of complexes
    extending the identity morphism on $G$, see Section~\ref{ssectcomp}.
    Then, $b \circ a : (F,\psi) \to (F,\psi)$ extends the identity morphism on $G$. Since the
    identity morphism on $(F,\psi)$ trivially also extends the identity morphism on $G$, we
    get by Proposition~\ref{propcomplexcomparison} that there exists $s_0 : F_0 \to F_1$ such that
    \begin{equation} \label{eqhomotopy}
        \Id_{F_0} = b_0 a_0 + \psi_1 s_0.
    \end{equation}

    We let $W = \tr (D\psi_1 \cdots D\psi_p R^F_p)$.
    We then get by \eqref{eqannright} and \eqref{eqhomotopy} that
    \begin{equation} \label{eqwdefeq}
        W = \tr (D\psi_1 \cdots D\psi_p R^F_p) = \tr (D\psi_1 \cdots D\psi_p R^F_p b_0 a_0),
    \end{equation}
    and by \eqref{eq:trABBA},
    \begin{equation*}
        W = \tr (a_0 D\psi_1 \cdots D\psi_p R^F_p b_0).
    \end{equation*}

    By the comparison formula \eqref{eq:comparisoncodimp}, applied to $b : (E,\varphi) \to (F,\psi)$, and Lemma~\ref{lmampplus1},
    \begin{equation*}
        R^F_p b_0 = b_p R^E_p - \psi_{p+1} \sigma^F_{p+1} b_p R^E_p,
   \end{equation*}
   where $\sigma^F_{p+1}$ is smooth.
   Since $D\psi_1 \cdots D\psi_p \psi_{p+1} = \psi_1 D\psi_2 \cdots
   D\psi_{p+1}$, see the previous proof, we get that
   \begin{equation*}
       W = \tr (a_0 D\psi_1 \cdots D\psi_p b_p R^E_p) - \tr(a_0 \psi_1 \alpha' R^E_p),
   \end{equation*}
   where $\alpha'$ is smooth.
   Thus, by Lemma~\ref{lma:DpsiDphi},
   \begin{equation} \label{eqw3}
       W = \tr (a_0 b_0 D\varphi_1 \cdots D\varphi_p R^E_p) + \tr (a_0 \psi_1 (\alpha-\alpha') R^E_p) + \tr(a_0 \beta \varphi_p R^E_p ).
   \end{equation}
   The last term in the right-hand side of \eqref{eqw3} vanishes by \eqref{eqannleft}.
   In addition, since $a$ is a morphism of complexes,
   $a_0 \psi_1 = \varphi_1 a_1$, so the middle term in the right-hand side of \eqref{eqw3} vanishes
   by \eqref{eqannR}.
   Thus, only the first term in the right-hand side of \eqref{eqw3} remains, i.e.,
   \begin{equation*}
       W = \tr (a_0 b_0 D\varphi_1 \cdots D\varphi_p R^E_p).
   \end{equation*}
   From \eqref{eq:trABBA} and \eqref{eqwdefeq} (with the roles of $(E,\varphi)$ and $(F,\psi)$ reversed),
   we finally conclude that
   \begin{equation*}
       W = \tr (D\varphi_1 \cdots D\varphi_p R^E_p).
   \end{equation*}
\end{proof}

\begin{proof} [Proof of Lemma~\ref{lmaindep2}]
Since the currents in \eqref{eqindep2} are pseudomeromorphic
$(p,p)$-currents, we may as in the previous proof assume that $G$ is
Cohen-Macaulay. 
    In addition, it is enough to prove \eqref{eqindep2} under the assumption that one of the
    free resolutions, say, $(F,\psi)$, has minimal length, $p$.
    We let $a : (F,\psi) \to (E,\varphi)$ and $b : (E,\varphi) \to (F,\psi)$ be morphisms
    of complexes extending the identity morphism on $G$.

    We claim that
    \begin{equation} \label{eq:Retainvers}
        R^F_p \eta^{-1} = b_p R^E_p \tau^{-1}.
    \end{equation}
    To see this, let $g \in \Ok_{Z,\zeta}$, and let $g_0$ be such that $\tau g_0 = g$. Then, by definition,
    \begin{equation} \label{eq:Rtauinvers}
        b_p R^E_p \tau^{-1} g = b_p R^E_p g_0,
    \end{equation}
cf.\ the text right before Theorem ~\ref{thmmain2}. 
    By \eqref{eq:comparisoncm}, the right-hand side of \eqref{eq:Rtauinvers}
    equals $R^F_p b_0 g_0$. Since $b$ extends the identity morphism, $\eta b_0 g_0 = \tau g_0 = g$.
    Thus, $R^F_p b_0 g_0$ equals by definition $R^F_p \eta^{-1} g$, which proves the claim.

    By \eqref{eq:Retainvers},
    \begin{equation} \label{eq:w5}
        \eta D\psi_1 \cdots D\psi_p R^F_p \eta^{-1} = \eta D\psi_1 \cdots D\psi_p b_p R^E_p \tau^{-1}.
    \end{equation}
    By Lemma~\ref{lma:DpsiDphi}, the right-hand side of \eqref{eq:w5} equals
    \begin{equation} \label{eq:w6}
        \eta b_0 D\varphi_1 \cdots D\varphi_p R^E_p \tau^{-1}
        + \eta  \psi_1 \alpha R^E_p \tau^{-1} + \eta\beta \varphi_p R^E_p \tau^{-1}.
    \end{equation}
    Since $\eta \psi_1 = 0$, the second term in the right-hand side of \eqref{eq:w6} vanishes,
    and the last term also vanishes by \eqref{eqannleft}. To conclude, using that $\eta b_0 = \tau$,
    we thus get \eqref{eqindep2}.
\end{proof}

By Lemma~\ref{lmaindep} and Lemma~\ref{lmaindep2},
\begin{equation*}
    \tr (D\varphi_1 \cdots D\varphi_p R^E_p) \text{ and } \tau D\varphi_1 \cdots D\varphi_p R^E_p \tau^{-1}
\end{equation*}
only depend on $G$ and not on the choice of free resolution
$(E,\varphi)$ of $G$ and connection $D$.
When $\rank E_0 = 1$, these currents coincide. If $G = \Ok_Z$, then there always exists a free resolution
$(F,\psi)$ of $\Ok_Z$ with $\rank F_0 = 1$, and thus, we get that for any free resolution $(E,\varphi)$ of $\Ok_Z$,
\begin{equation} \label{eqtraceidentification}
    \tr(D\varphi_1 \cdots D\varphi_p R^E_p) = \tau D\varphi_1 \cdots D\varphi_p R^E_p \tau^{-1}.
\end{equation}

\begin{proof}[Proof of Theorem~\ref{thmmain2}]
Note that by \eqref{eqtraceidentification}, it is enough to prove \eqref{eqmain2}.

    Let $W = Z_{\rm red}$. We
    first consider a point $\xi \in W_{\rm reg}$, and apply Lemma~\ref{lmagoodcoord2}.
    We fix a neighbourhood $V \subset X$ of $\xi$ contained in the coordinate chart from Lemma~\ref{lmagoodcoord2}
    such that $W = \{ z_1 = \dots = z_p = 0 \}$ on $V$, and $V \cap W = U$.
    We first prove that \eqref{eqmain2} holds on $V$. 
Note that on $V$, $[Z] = m[z_1 = \dots = z_p = 0 ]$,
    so we thus want to prove that
    \begin{equation} \label{eqlocalnicemain}
        \frac{1}{(2\pi i)^p p!}\tr ( D\varphi_1 \cdots D\varphi_p R_p^E ) = m[z_1 = \dots = z_p = 0 ].
    \end{equation}

    \begin{lma} \label{lma:mainKoszul}
        Let $\zeta \in U \setminus Y$, and let $(K,\phi)$ be the universal free resolution of $\Ok_{Z,\zeta}$ from 
        Theorem~\ref{thmkoszul}. Then
    \begin{equation} \label{eqlocalnicemainK}
        \frac{1}{(2\pi i)^p p!}\tr ( D\phi_1 \cdots D\phi_p R_p^K ) =
        m[z_1 = \dots = z_p = 0 ] 
    \end{equation}
        in a neighbourhood of $\zeta$.
    \end{lma}
    Taking this lemma for granted, using Lemma~\ref{lmaindep} and Theorem~\ref{thmkoszul},
    we get first that \eqref{eqlocalnicemain} holds in a neighbourhood of each $\zeta \in U\setminus Y$.
    Thus, \eqref{eqlocalnicemain} holds in a neighbourhood of $U \setminus Y$, and since
    both sides of \eqref{eqlocalnicemain} have their support on $V \cap W = U$,
    \eqref{eqlocalnicemain} holds in fact on $V \setminus Y$.
    Since $Y$ is a hypersurface of $W$, and $W$ has codimension $p$ in $V$, $Y$ has codimension $p+1$ in $V$.
    As both sides of \eqref{eqlocalnicemain} are pseudomeromorphic $(p,p)$-currents on $V$ which coincide outside
    of $Y$, \eqref{eqlocalnicemain} holds on all of $V$ by the dimension principle, Proposition~\ref{proppmdim}.

    We have thus proven that any point $\xi \in W_{\rm reg}$ has a neighbourhood
    such that \eqref{eqmain2} holds, and since both sides of \eqref{eqmain2} have support on $W$,
    \eqref{eqmain2} holds on $X \setminus W_{\rm sing}$.
    Both sides of \eqref{eqmain2} are pseudomeromorphic $(p,p)$-currents on $X$, and
    $W_{\rm sing}$ has codimension $\geq p+1$ in $X$, so we get by the dimension principle
    that \eqref{eqmain2} holds on all of $X$.
\end{proof}

\begin{proof} [Proof of Lemma~\ref{lma:mainKoszul}] 
    We here use the notation from Section~\ref{sectkoszul}, and we let $e_1,\dots,e_p$ be the standard basis for
    $A^{\oplus p}$ over $A$. Note that over $\Ok_{X,\zeta}$, $\bigwedge^k A^{\oplus p}$ has the basis
    $\cl{z^{\alpha^i}} e_I$, where $i=1,\dots,m$ and $I \subset \{ 1,\dots,p\}$, $|I| = k$.
    Since by Lemma~\ref{lmaindep}, the left-hand side of \eqref{eqlocalnicemainK} is independent of the choice of connection,
    we may assume that $D$ is trivial with respect to these bases.

    In order to prove \eqref{eqlocalnicemainK}, we first write out the left-hand side as
    \begin{equation} \label{eqtraceexpression}
        \tr (D\phi_1 \cdots D\phi_p R^K_p) = \sum_{i=1}^m (\cl{z^{\alpha^i}}e_\emptyset)^* D\phi_1\cdots D\phi_p R^K_p \cl{z^{\alpha^i}} e_\emptyset,
    \end{equation}
    where
    $(\cl{z^{\alpha^1}}e_\emptyset)^*,\dots,(\cl{z^{\alpha^m}}e_\emptyset)^*$
    is the dual basis of 
    the basis $\cl{z^{\alpha^1}} e_\emptyset,\dots,\cl{z^{\alpha^m}} e_\emptyset$ of $K_0$.
    
    We will use the comparison formula, Theorem ~\ref{thmRcomparison}, to compute the currents $R^K_p \cl{z^{\alpha^i}} e_\emptyset$ appearing in the sum 
    in the right-hand side of \eqref{eqtraceexpression}.
    First of all, by the Nullstellensatz, there exist $\beta_i$ such that $z_i^{\beta_i} \in \mathcal{J}$ for $i = 1,\dots,p$.
    Throughout this proof, we will let $\beta_1,\dots,\beta_p$ denote such a choice.
    We let $\epsilon_1,\dots,\epsilon_p$ be the standard basis of $\Ok_{X,\zeta}^{\oplus p}$ over $\Ok_{X,\zeta}$.
    We let $(L,\psi)$ be the Koszul complex over $\Ok_{X,\zeta}$ of the tuple
    $(z_1^{\beta_1},\dots,z_p^{\beta_p})$, and we let $\mathcal{I}$ be the ideal generated by this tuple.

    Since $\mathcal{I}$ is contained in $\mathcal{J}$, there exists a morphism
    $c : (L,\psi) \to (K,\phi)$ extending the natural surjection
    $\Ok_{X,\zeta}/\mathcal{I} \to \Ok_{Z,\zeta}$, see Proposition
    ~\ref{propcomplexcomparison}. 
    We construct explicitly such a morphism $c$.
    We let $c_k$ be the map $L_k = \bigwedge^k \Ok_{X,\zeta}^{\oplus p} \to \bigwedge^k A^{\oplus p} = K_k$
    induced by the map $c_1 : \Ok_{X,\zeta}^{\oplus p} \to A^{\oplus p}$,
    \begin{equation*}
        c_1 : \epsilon_i \mapsto \sum_{\gamma_i=0}^{\beta_i-1} z_i^{\beta_i-\gamma_i-1}\cl{z_i^{\gamma_i}} e_i,
    \end{equation*}
    i.e., $c_k$ is defined by
    \begin{equation*}
        c_k : \epsilon_{i_1} \wedge \cdots \wedge \epsilon_{i_k} \mapsto c_1(\epsilon_{i_1})\wedge\cdots\wedge c_1(\epsilon_{i_k}).
    \end{equation*}
    Here, $c_0 : L_0  \to K_0$ is to be interpreted as $\epsilon_\emptyset \mapsto \cl{1} e_\emptyset$.
    It is straightforward to check that $c$ is a morphism of complexes extending the natural surjection
    $\Ok_{X,\zeta}/\mathcal{I} \to \Ok_{Z,\zeta}$ by using the formula
    \begin{equation*}
        (z_j - \cl{z_j})\left(\sum_{\gamma_j = 0}^{\beta_j-1} z_j^{\beta_j-\gamma_j-1} \cl{z_j^{\gamma_j}} \right)
        = z_j^{\beta_j}\cl{1} - \cl{z_j^{\beta_j}} = z^{\beta_j}_j\cl{1},
    \end{equation*}
    where the last equality comes from that $z_j^{\beta_j} = 0$ in $\Ok_{Z,\zeta}$.

    We now fix some $i \in \{ 1,\dots,m\}$, and let $\tilde{c} := (\cl{z^{\alpha^i}} c) : (L,\psi) \to (K,\phi)$
    (i.e., $\tilde c$ equals $c$ composed with multiplication with $\cl{{z^{\alpha^i}}}$).
    This is clearly a morphism of complexes, with $\tilde{c}_0(\epsilon_\emptyset) = \cl{z^{\alpha^i}} e_\emptyset$.
    Thus, using the comparison formula, \eqref{eq:comparisoncm}, for $\tilde{c}$,
    \begin{equation*}
        R^K_p \cl{z^{\alpha^i}} e_\emptyset \epsilon_\emptyset^* = \cl{z^{\alpha^i}} c_p R^L_p.
    \end{equation*}
    Applying this to each term in the sum in \eqref{eqtraceexpression}, we get that
    \begin{multline*}
        \tr (D\phi_1 \cdots D\phi_p R^K_p) = \sum e_\emptyset^*\cl{z^{\alpha^i}}^* D\phi_1\cdots D\phi_p R^K_p \cl{z^{\alpha^i}}
        e_\emptyset \epsilon_\emptyset^*  \epsilon_\emptyset = \\ 
        =  \sum e_\emptyset^*\cl{z^{\alpha^i}}^* D\phi_1\cdots D\phi_p \cl{z^{\alpha^i}} c_p R^L_p \epsilon_\emptyset.
    \end{multline*}
    We write the map $c_p$ as
    \begin{equation*}
        c_p : \epsilon_{\{1,\dots,p\}} \mapsto \tilde{B} \wedge e_{\{1,\dots,p\}},
    \end{equation*}
    where
    \begin{equation*}
        \tilde{B} = \sum_{\gamma \leq \beta-{\bf 1}} z^{\beta-\gamma-{\bf 1}} \cl{z^{\gamma}}.
    \end{equation*}
    Since $\cl{z^{\alpha^i}}$ and $\tilde{B}$ commute, being elements of $A$, we get that
    \begin{align*}
        \tr (D\phi_1 \cdots D\phi_p R^K_p) = \sum e_\emptyset^*\cl{z^{\alpha^i}}^* D\phi_1\cdots D\phi_p \tilde{B} \cl{z^{\alpha^i}}
        e_{\{1,\dots,p\}} \epsilon_{\{1,\dots,p\}}^* R^L_p \epsilon_\emptyset.
    \end{align*}
    We let $B$ be the form-valued $\Ok_{X,\zeta}$-linear map $A \to A$ given by
    \begin{equation*}
        B := e_\emptyset^* D\phi_1 \cdots D\phi_p \tilde{B} e_{\{1,\dots,p\}}.
    \end{equation*}
    Using that $e_\emptyset^*$ and $\cl{z^{\alpha^i}}^*$ commute, and that $e_{\{1,\dots,p\}}$ and $\cl{z^{\alpha^i}}$ commute,
    we then get that 
    \begin{equation*}
        \tr (D\phi_1 \cdots D\phi_p R^K_p) = 
        \sum \cl{z^{\alpha^i}}^* B \cl{z^{\alpha^i}} \epsilon_{\{1,\dots,p\}}^* R^L_p \epsilon_\emptyset = 
        (\tr B) \epsilon_{\{1,\dots,p\}}^* R^L_p \epsilon_\emptyset.
    \end{equation*}
    Note that by \eqref{eq:bmch} and \eqref{circa},
    \begin{equation*}
        \epsilon_{\{1,\dots,p\}}^* R^L_p \epsilon_\emptyset = (-1)^{p^2}\dbar \frac{1}{z_p^{\beta_p}} \wedge \cdots \wedge \dbar \frac{1}{z_1^{\beta_1}}.
    \end{equation*}
Moreover, in view of the Poincar\'e-Lelong formula
\eqref{eqpl}, note that 
\begin{equation*} 
    (-1)^{p^2} \frac{1}{(2\pi i)^p} z^{\beta-{\bf 1}}  dz_1 \wedge \dots \wedge
    dz_p \wedge \dbar \frac{1}{z_p^{\beta_p}} \wedge \dots \wedge \dbar \frac{1}{z_1^{\beta_1}} 
    = [z_1 = \dots = z_p = 0]. 
\end{equation*}
  Thus, from Lemma~\ref{lmadphiapart} below, we conclude that
    \eqref{eqlocalnicemainK} holds.
\end{proof}

\begin{lma} \label{lmadphiapart}
    Let $B$ be as in the proof of Lemma~\ref{lma:mainKoszul}.
    Then
    \begin{equation} \label{eqtrM}
        \tr B = p! m z^{\beta-{\bf 1}} dz_1 \wedge \cdots \wedge dz_p.
    \end{equation}
\end{lma}

\begin{proof}
    As $\phi_k$ is contraction with $\mathbf{z}_1 e_1 + \dots + \mathbf{z}_p e_p$, and $D$ is assumed to be trivial
    with respect to the bases $\cl{z^{\alpha^i}} e_I$, we get in the same way as in Example~\ref{exDkoszul} that
    \begin{equation*}
        e_\emptyset^* D\phi_1 \cdots D\phi_p e_{\{1,\dots,p\}} =
        p! D\mathbf{z}_1 \cdots D\mathbf{z}_p.
    \end{equation*}
    Since $\mathbf{z}_i = z_i - \cl{z_i}$, we thus get that
    $B$ is a sum of terms of the form
    \begin{equation} \label{eqMterm}
        \pm p! dz_{I} \wedge (D\cl{z_{J_1}})\cdots (D\cl{z_{J_\ell}}) z^{\beta-\gamma-{\bf 1}} \cl{z^\gamma},
    \end{equation}
    where $|I|+|J| = p$, and $I \cup J = \{ 1,\dots,p\}$.
   
 We claim that the traces of all such terms are zero, unless $|J| = 0$ and $\gamma = 0$.
    Recall from Remark~\ref{rem:upper-triangular} that, in the basis of $A$ given by $\cl{z^{\alpha^1}},\dots,\cl{z^{\alpha^m}}$,
    the matrix for multiplication with any monomial $\cl{z^{\delta}}$ is upper triangular, and in addition,
    it will have zeros on the diagonal if and only if $\delta \neq 0$.
    Thus, the matrix of each $D\cl{z_{J_i}}$ is a (form-valued) upper triangular matrix with zeros on the diagonal,
    since $D$ is assumed to be trivial with respect to the bases $\cl{z^{\alpha^i}}e_I$.
    Since $\cl{z^{\gamma}}$ is also upper-triangular, the full product \eqref{eqMterm} is upper-triangular,
    and with zeros on the diagonal if $|J| > 0$ or $\gamma \neq 0$. Thus, the trace is zero in case $|J| > 0$ or $\gamma \neq 0$,
    which proves the claim. 

To conclude,
    \begin{equation*}
        \tr B = p! dz_1 \wedge \dots \wedge dz_p z^{\beta-\mathbf{1}} \tr \cl{1},
    \end{equation*}
    and since $\tr \cl{1} = \rank_{\Ok_{X,\zeta}} A = m$, we obtain \eqref{eqtrM}.
\end{proof}

\begin{proof}[Proof of Theorem~\ref{thmnonpure}]
    We let $[Z]_{[k]}$ be the part of the fundamental cycle $[Z]$ of codimension $k$,
    i.e., $[Z]_{[k]} = \sum m_i [Z_i]$, where the sum is over the irreducible components $Z_i$
    of $Z_{\rm red}$ of codimension $k$, and $m_i$ is the geometric multiplicity of $Z_i$ in $Z$.
    Thus,
    \begin{equation*}
        [Z] = \sum_k [Z]_{[k]},
    \end{equation*}
    and it is enough to prove that
    \begin{equation} \label{eq:codimk}
       \frac{1}{(2\pi i)^k k!} D\varphi_1 \cdots D\varphi_k R_{[k]} = [Z]_{[k]},
    \end{equation}
    for $k=\codim Z,\dots,N$.
    Let $V_k = W_k \cap (\cup_{q \neq k} W_q)$; then $V_k$ has codimension $\geq k+1$.
    Note that both sides of \eqref{eq:codimk} have support on $W_k$, and that
    $Z$ has pure codimension $k$ on $W_k \setminus V_k$. Thus, \eqref{eq:codimk} holds
    on $X\setminus V_k$ by Theorem~\ref{thmmain}. Since $\codim V_k \geq k+1$ and both
    sides of \eqref{eq:codimk} are pseudomeromorphic $(k,k)$-currents, \eqref{eq:codimk} holds everywhere
    by the dimension principle, Proposition~\ref{proppmdim}.
\end{proof}

\begin{remark}\label{duo}
By analogous arguments we can prove \eqref{sprak}. 
First 
\begin{equation*}
    \frac{1}{(2\pi i)^kk!} \tr(D\varphi_1 \cdots D\varphi_k
    R_{[k]}) = 
\frac{1}{(2\pi i)^kk!} \tau D\varphi_1 \cdots D\varphi_k
    R_{[k]} \tau^{-1}=
[Z]_{[k]},
\end{equation*}
holds on $X\setminus V_k$ by Theorem ~\ref{thmmain2} and thus it holds
everywhere by the dimension principle. 
\end{remark}

\section{Examples of higher degree currents}\label{intressant}

We will start by illustrating 
Theorem ~\ref{thmnonpure} by explicitly computing the left-hand
side of \eqref{jesse} in the situation of Example ~\ref{condo}.

\begin{ex}
Let $Z$ be as in Example ~\ref{condo}. Then $\Ok_Z$ has a (minimal) free resolution
    \begin{equation*}
        0 \to \Ok_{\mathbb C^3} \stackrel{\varphi_2}{\to} \Ok_{\mathbb
          C^3}^{\oplus 2} \stackrel{\varphi_1}{\to} \Ok_{\mathbb C^3},
    \end{equation*}
    where
\begin{equation*}
    \{\varphi_2\} = \left[ \begin{array}{c} -y \\ x \end{array} \right] \text { and }
       \{ \varphi_1 \}= \left[ \begin{array}{cc} xz & yz \end{array}
        \right]. 
\end{equation*} 
Let $D$ be (induced by) the trivial connections on
$E_0=\Ok_{\C^3}$, $E_1=\Ok_{\C^3}^{\oplus 2}$, and $E_2=\Ok_{\C^3}$. 
In \cite{L15}*{Example~5}, the current $R^E=R^E_1+R^E_2$ was computed explicitly: 
\begin{eqnarray*}
&& \{R^E_1\} = \frac{1}{|x|^2+|y|^2}\left[\begin{array}{c} \bar{x}\\\bar{y} \end{array}\right]\dbar\frac{1}{z}\\
&& \{R^E_2\}  = \frac{1}{z} \dbar\frac{1}{y}\wedge \dbar\frac{1}{x} +  
            \dbar\left(\frac{\left[\begin{array}{cc} -\bar{y} & \bar{x} \end{array}\right]}{|x|^2+|y|^2}\right)
                \frac{1}{|x|^2+|y|^2}\left[\begin{array}{c}
                    \bar{x}\\\bar{y} \end{array}\right]\wedge\dbar\frac{1}{z}=:
\mu_1 + \mu_2.
\end{eqnarray*}

Note that the irreducible components $Z_1:=\{z=0\}$ and
$Z_2:=\{x=y=0\}$ of $Z$ are of codimension $1$ and $2$, respectively; thus 
$R^E_{[k]}=\mathbf 1_{Z_k} R_k^E$ for $k=1,2$. 
Since $R^E_1$ has support on $Z_1$ it follows that
$R^E_{[1]}=R^E_{1}$. 
Since $\supp \mu_2 \subseteq \{ z = 0 \}$, ${\mathbf 1}_{Z_2} \mu_2$ has support
on $Z_2 \cap \{ z = 0 \}=\{x=y=z=0\}$, which has codimension $3$, and thus it vanishes by the 
dimension principle. Since $\supp \mu_1 \subseteq \{ x = y = 0 \} = Z_2$,
we get that ${\mathbf 1}_{Z_2} \mu_1 = \mu_1$. Thus, to conclude,
\begin{equation*}
    \{ R^E_{[1]} \} = \frac{1}{|x|^2+|y|^2}\left[\begin{array}{c} \bar{x}\\\bar{y} \end{array}\right]\dbar\frac{1}{z} \text{ and }
    \{ R^E_{[2]} \} = \frac{1}{z} \dbar\frac{1}{y}\wedge \dbar\frac{1}{x}. 
\end{equation*}
By a straightforward calculation, one can then verify \eqref{jesse} in this case. 
\end{ex} 

It would be interesting to consider the full currents 
\begin{equation}\label{yoga}
D\varphi_1\cdots D\varphi_k R^E_k
\end{equation}
(and not only $D\varphi_1\cdots D\varphi_k R^E_{[k]}$) and investigate 
whether these capture algebraic or geometric information (in addition
to the fundamental cycle). 
If $(E, \varphi)$ is the Koszul complex of a holomorphic tuple $f$ it
was shown in \cite{ASWY} that the currents \eqref{yoga} satisfy a
generalized King's formula, generalizing \cite{ALelong}; in
particular, the Lelong numbers are 
the so-called \emph{Segre numbers} of the ideal generated by $f$. 

We remark that in the above example we do not know how to
interpret the current 
$D\varphi_1D\varphi_2 R^E_2$ or rather the part 
$D\varphi_1D\varphi_2\mu_2$. 
Below, however, we will consider an example where
$D\varphi_1D\varphi_2 R^E_2$ is a current of integration along the
(only) associated prime of codimension 2. 
For an ideal $\mathcal J$ over a local ring $R$, there is a notion of the
\emph{length along an associated prime}
$\mathfrak p$, defined as the length of the largest ideal in
$R_\mathfrak p/\mathcal JR_\mathfrak p$ of finite
length, see for example \cite{EH}*{Sect. II.3, p. 68}. 
The length of $\mathcal J$ along $\mathfrak p$ coincides with the geometric
multiplicity of $\mathcal J(\mathfrak p)$ in $\mathcal J$ if $\mathfrak p$ is a
minimal associated prime of $\mathcal J$. 
It would be interesting
to see whether these numbers could be recovered from the currents
\eqref{yoga}. However, in view of the example below this is not clear how to
do.

\begin{ex}
Let $Z$ be as in Example ~\ref{midair}. 
Then
\begin{equation*}
    0\to \Ok_{\C^2} \stackrel{\varphi_2}{\longrightarrow}
    \Ok_{\C^2}^{\oplus 2} \stackrel{\varphi_1}{\longrightarrow}
    \Ok_{\C^2} \to \Ok_Z, 
\end{equation*}
where
\begin{equation*}
    \{\varphi_2\} = \left[ \begin{array}{c} -x^\ell \\ y^{k-m} \end{array} \right] \text { and }
        \{\varphi_1\} = \left[ \begin{array}{cc} y^k & x^\ell
            y^m \end{array} \right], 
\end{equation*} 
is a free resolution of $\Ok_Z$. 
Note that, since $Z_\text{red}$ only has one irreducible component $\{y=0\}$ of
codimension $1$, $R^E_{[2]}=0$. 

Let $D$ be (induced by) the trivial connections on
$E_0=\Ok_{\C^2}$, $E_1=\Ok_{\C^2}^{\oplus 2}$, and $E_2=\Ok_{\C^2}$. 
Then a direct computation  yields 
\begin{equation*}
    \{D\varphi_1 D\varphi_2 \}= -\ell(2k-m) x^{\ell-1}y^{k-1} dx\wedge dy
    =: - C x^{\ell-1}y^{k-1} dx\wedge dy, 
\end{equation*}
where, as above, we have used the notation from Section
~\ref{matrisnot}. 
Next, let $(F,\psi)$ be the Koszul complex of $(y,x)$
and let
$a_0 : F_0 \to E_0$ be given by $\{a_0\} = \left[\begin{array}{c} x^{\ell-1}y^{k-1}\end{array}\right]$. 
Then 
$\{R^F_2\}=\dbar(1/x)\wedge \dbar(1/y)$
and $a_0$ can be extended to a morphism of complexes $a : (F,\psi) \to (E,\varphi)$,
where
\begin{equation*}
    \{a_2\} = \left[ \begin{array}{c} 1 \end{array} \right] \text { and }
        \{a_1\} = \left[ \begin{array}{cc} x^{\ell-1} & 0 \\ 0 & y^{k-m-1} \end{array} \right].
\end{equation*}
If we apply the comparison formula, \eqref{eq:comparisongeneral}, and
identify the components that takes values in $\Hom(F_0,E_2)$ we get
that 
\begin{equation*}
    R^E_2 a_0 - a_2 R^F_2 = \varphi_3 M_3 -\dbar M_2. 
\end{equation*}
Note that $M_3=0$ since $(E,\varphi)$ has length $2$. Moreover, since $Z^E_2 = Z^F_1 = \{ x = y = 0 \}$
has codimension $\geq 2$, $M_2 = 0$ by \cite{LComp}*{Proposition~3.5}. 
Hence $R^E_2 a_0 = a_2 R^F_2$. 
Thus, we get that
\begin{multline*}
\{D\varphi_1 D\varphi_2 R^E_2 \}= 
-C x^{\ell-1} y^{k-1} dx \wedge dy \{R^E_2 \}=
-C dx \wedge dy \{R^E_2 a_0\} = \\
-C dx \wedge dy \{a_2 R^F_2\} =
-C dx \wedge dy \wedge \dbar\frac{1}{x}\wedge \dbar \frac{1}{y}
 = (2\pi i)^2 C [0],
\end{multline*}
cf. \eqref{eqmatrixmultsign}. 

We conclude that 
\begin{equation}\label{poll}
D\varphi_1 D\varphi_2 R^E_2 = (2\pi i)^2 \ell(2k-m) [0],  
\end{equation}
i.e., $D\varphi_1 D\varphi_2 R^E_2$ is the current of integration
along the (only) associated prime $\mathfrak{m}_{\C^2,0} =
\mathcal{J}(x,y)$ of $\mathcal J$ with mass $(2\pi i)^2 \ell(2k-m)$. 
However, a computation yields that the length of $\mathcal J_0$ along
$\mathfrak{m}_{\C^2,0}$ equals $\ell(k-m)$; it it not clear to us how
to relate these numbers. 
\end{ex}

\section{Relation to the results of Lejeune-Jalabert} \label{sectlj}

Our results are closely related to results by Lejeune-Jalabert,
\cite{LJ1, LJ2}, and we will in this section compare
our results with hers.

Throughout this section, we let $Z$ be a (not necessarily reduced) analytic space of pure dimension $n$.
Assume that $Z$ is a subspace of codimension $p$ of the complex manifold $X$ of dimension $N = n+p$,
and let $Z$ be defined by the ideal sheaf $\mathcal{J} \subset \Ok_X$.

\subsection{The Grothendieck dualizing sheaf and residue currents} \label{ssectdualizing}

If $Z$ is Cohen-Macaulay, then the \emph{Grothendieck dualizing sheaf} $\omega_Z$ is
\begin{equation*}
    \omega_Z := \Ext^p_{\Ok_X}(\Ok_Z,\Omega_X^N),
\end{equation*}
where $\Omega_X^N$ is the sheaf of holomorphic $N$-forms on $X$.
If $Z$ is smooth, then $\omega_Z$ coincides with $\Omega_Z$.

One way of realizing $\omega_Z$ is as $H^p\big (\SHom(E_\bullet,\Omega^N_X)\big)$,
where $(E,\varphi)$ is a locally free resolution of $\Ok_Z$,
and another is as $H^p\big (\SHom(\Ok_Z,\mathcal{C}^{N,\bullet})\big )$,
where $(\mathcal{C}^{N,\bullet},\dbar)$ is the Dolbeault complex of $(N,\bullet)$-currents
on $X$.
There is a canonical isomorphism between these representations of $\omega_Z$,
\begin{equation} \label{eqdualisom}
    \res : H^p\big(\SHom(E_\bullet,\Omega^N_X)\big) \stackrel{\cong}{\to} H^p\big(\SHom(\Ok_Z,\mathcal{C}^{N,\bullet})\big),
\end{equation}
and by \cite{AndNoeth}*{Theorem~1.5 and Example~1}, this isomorphism can be realized concretely by the
residue current\footnote{We have introduced the factor $1/(2\pi i)^p$ for normalization reasons.} $R^E_p$:
\begin{equation} \label{eqdualisom2}
    \res : [\xi] \mapsto \left[\frac{1}{(2\pi i)^p} \xi R^E_p\tau^{-1} \right],
\end{equation}
where $\tau$ is the natural surjection $\tau : E_0 \to \coker \varphi_1 \cong \Ok_Z$ and we consider
$\xi R^E_p\tau^{-1}$ as a scalar current in a similar way as in the introduction.

\subsection{Coleff-Herrera currents}

A $(q,p)$-current $\mu$ on $X$ is a \emph{Coleff-Herrera current} on $Z_{\rm red}$, denoted $\mu \in \CH^q_{Z_{\rm red}}$,
if $\dbar \mu = 0$, $\overline{\psi} \mu = 0$ for all holomorphic functions $\psi$ vanishing
on $Z_{\rm red}$, and $\mu$ has the \emph{SEP with respect to} $Z_{\rm red}$, i.e.,
for any hypersurface $V$ of $Z_\text{red}$, the limit ${\bf 1}_V \mu := \lim_{\epsilon\to 0} (1-\chi(|f|/\epsilon)) \mu$ exists
and ${\bf 1}_V \mu = 0$,
where $f$ is a tuple of holomorphic functions defining $V$.
This description of Coleff-Herrera currents is due to Bj\"ork, see \cite{BjDmod}*{Chapter~3},
and \cite{BjAbel}*{Section~6.2}.

Let $\mathcal{G}$ be a coherent sheaf of codimension $p$, with a locally free resolution $(E,\varphi)$
of length $p$ (so that in particular, $\mathcal{G}$ is Cohen-Macaulay).
Then $R^E_p$ is a $\Hom(E_0,E_p)$-valued Coleff-Herrera current on $V := \supp \mathcal{G}$.
To see this, note first that, by the $\nabla$-closedness of $R^E$ and the fact that $E$ has length $p$,
$\dbar R^E_p = \varphi_{p+1} R^E_{p+1} = 0$. The fact that $R^E_p$ has the SEP follows from the dimension principle,
Proposition~\ref{proppmdim}. Moreover that $\overline{\psi} R^E_p = 0$ for any holomorphic function $\psi$ vanishing on $V$
follows from the fact that $R^E_p$ is a pseudomeromorphic current with support on $V$, see \cite{AW2}*{Proposition~2.3}.

We let $(\mathcal{C}^{N,\bullet}_{[Z_{\rm red}]},\dbar)$ denote the Dolbeault complex of $(N,\bullet)$-currents on $X$
with support on $Z_{\rm red}$.
It was proven in \cite{DS1} (for $Z_{\rm red}$ a complete intersection) and \cite{DS2}*{Proposition~5.2} (for $Z_{\rm red}$ arbitrary
of pure dimension) that Coleff-Herrera currents are canonical representatives in moderate cohomology in the sense that
\begin{equation*}
    \left(\ker \dbar : \mathcal{C}^{N,p}_{[Z_{\rm red}]} \to \mathcal{C}^{N,p+1}_{[Z_{\rm red}]}\right) \cong
    \CH^N_{Z_{\rm red}} \oplus \dbar \mathcal{C}^{N,p-1}_{[Z_{\rm red}]},
\end{equation*}
i.e., each cohomology class in $H^p\big
(\mathcal{C}^{N,\bullet}_{[Z_{\rm red}]}\big )$ has a unique representative
which is a Coleff-Herrera current.
In particular, 
\begin{equation} \label{eq:ch-uniqueness}
    \CH^N_{Z_{\rm red}} \cap \left(\im \dbar : \mathcal{C}^{N,p-1}_{[Z_{\rm red}]} \to \mathcal{C}^{N,p}_{[Z_{\rm red}]}\right) = \{ 0 \}.
\end{equation}

\subsection{Relation to the results in \cite{LJ1}}

In this section, we discuss how the results of Lejeune-Jalabert give our results and vice versa.
The main point is to describe how the result of \cite{LJ1} give the following special case of Theorem~\ref{thmmain2}.

\begin{thm} \label{thmcm}
    Let $Z \subset X$ be an analytic space of pure codimension $p$ which is Cohen-Macaulay.
    Assume that $\Ok_Z$ has a locally free resolution $(E,\varphi)$ over $\Ok_X$ of length $p$,
    and let $D$ be the connection on $\End E$ induced by connections on $E_0,\dots,E_p$.
    Then,
    \begin{equation} \label{eqcm1}
        \frac{1}{(2\pi i)^p p!} \tr(D\varphi_1 \cdots D\varphi_p R^E_p) = [Z],
    \end{equation}
    and
    \begin{equation} \label{eqcm2}
        \frac{1}{(2\pi i)^p p!} \tau D\varphi_1 \cdots D\varphi_p R^E_p \tau^{-1} = [Z],
    \end{equation}
    where $\tau$ is the natural surjection $\tau : E_0 \to \coker \varphi_1 \cong \Ok_Z$.
\end{thm}

In order to prove Theorem~\ref{thmmain2} in full generality,
without assuming that $Z$ is Cohen-Macaulay or that $(E,\varphi)$ has length $p$, one can then argue
in the same way as in our proof of Theorem~\ref{thmmain2}, 
but using Theorem~\ref{thmcm} instead of Lemma~\ref{lma:mainKoszul}. Indeed, first of all,
by \eqref{eqtraceidentification}, it is sufficient to prove just \eqref{eqmain2}.
By combining Lemma~\ref{lmaindep} and Theorem~\ref{thmcm}, we first obtain \eqref{eqmain2}
in a neighbourhood of each Cohen-Macaulay point.
By the dimension principle, \eqref{eqmain2} then holds on all of $X$.

\smallskip

In \cite{LJ1}, the fundamental class of $Z$ is considered as a map $c_Z : \Omega_Z^n \to \omega_Z$,
where $\Omega_Z^n$ is the sheaf of holomorphic $n$-forms on $Z$.
If $\alpha$ is a section of $\Omega_Z^n$ 
and $\tilde{\alpha}$ is a section of $\Omega_X^n$, which is a representative of $\alpha$, then
$\gamma := \tilde{\alpha} \wedge \tau D\varphi_1 \cdots D\varphi_p$ is a section of $\SHom(E_p,\Omega^N_X \otimes \Ok_Z)$.
Since $(E, \varphi)$ has length $p$, $\gamma$ induces a section $[\gamma]$ of $\Ext^p(\Ok_Z,\Omega_X^N\otimes \Ok_Z)$.
We now consider the isomorphism
\begin{equation} \label{eq:alj-isom}
    \omega_Z = \Ext^p(\Ok_Z,\Omega_X^N) \cong \Ext^p(\Ok_Z,\Omega_X^N\otimes \Ok_Z)
\end{equation}
induced by the surjection $\Omega_X^N \to \Omega_X^N \otimes \Ok_Z$,
see \cite{ALJ}*{Proposition~4.6}.
Since $E_p$ is locally free, $\gamma$ can locally be lifted to sections $\gamma_i$ of $\SHom(E_p,\Omega_X^N)$.
Since $(E,\varphi)$ has length $p$, these local liftings of $\gamma$ define sections $[\gamma_i]$ of $\omega_Z$
locally. On overlaps, the $\gamma_i$'s differ by sections of $\SHom(E_p,\Omega_X^N)\otimes \mathcal{J}$,
and since $\mathcal{J} \omega_Z = 0$, the sections $[\gamma_i]$ patch together to a global section
of $\omega_Z$, which we denote by $[\tilde \alpha \wedge \tau
D\varphi_1 \cdots D\varphi_p]$. 
By construction, $[\tilde \alpha \wedge \tau
D\varphi_1 \cdots D\varphi_p]$ maps to $[\gamma]$ using the isomorphism
\eqref{eq:alj-isom}. 
The main theorem in \cite{LJ1} asserts that this gives the fundamental
class of $\alpha$ (times $p!$), i.e.,
\begin{equation} \label{eqljmain}
    c_Z(\alpha) = \frac{1}{p!}[\tilde{\alpha} \wedge \tau D\varphi_1  \cdots  D\varphi_p].
\end{equation}

Note that where the local lifting $\gamma_i$ of $[\tilde{\alpha} \wedge \tau D\varphi_1  \cdots  D\varphi_p]$ is defined, $\gamma_i R^E_p$ coincides
with $\gamma R^E_p=\tilde{\alpha} \wedge \tau D\varphi_1\cdots D\varphi_p R^E_p$ (if
we consider the currents as scalar currents). 
Thus combining \eqref{eqljmain} with the realization
\eqref{eqdualisom2} of the isomorphism \eqref{eqdualisom}, we get that 
\begin{equation} \label{eqljmaincurrent}
    \res c_Z(\alpha) = \frac{1}{(2\pi i)^p p!}\tilde{\alpha} \wedge \tau D\varphi_1\cdots D\varphi_p R^E_p \tau^{-1} + \dbar \SHom(\Ok_Z,\mathcal{C}^{N,p-1}).
\end{equation}

It is not entirely clear to us how the fundamental
class is defined in \cite{LJ1}, but it is reasonable to assume that if one uses the isomorphism
\eqref{eqdualisom} to represent $c_Z(\alpha)$ as a current, then one should have
\begin{equation} \label{eqfundassumption}
    \res c_Z(\alpha) = \tilde{\alpha} \wedge [Z],
\end{equation}
where we by $[Z]$ mean the fundamental cycle (seen as a current on
$X$) as defined in \eqref{eqfundcycle}.
Since we have an independent proof of Theorem ~\ref{thmcm}
this assumption must indeed be correct, cf.\ the last paragraph
below. 
Note that since the right-hand side of \eqref{eqfundassumption} is a pseudomeromorphic $(p,p)$-current,
by the dimension principle, it is uniquely determined by its restriction to $Z_{\rm reg}$, and hence,
it is independent of the precise definition of $\Omega_Z^n$ as long as
the forms in $\Omega_Z^n$ coincide with regular holomorphic $n$-forms
on $Z_\text{reg}$ and can be lifted to holomorphic $n$-forms on
$X$. 

If we assume \eqref{eqfundassumption}, then \eqref{eqljmaincurrent} implies that
\begin{equation*}
    \mu := \tilde{\alpha} \wedge [Z] - \frac{1}{(2\pi i)^p p!}\tilde{\alpha} \wedge \tau D\varphi_1\cdots D\varphi_p R^E_p \tau^{-1} \in
    \left(\im \dbar : \mathcal{C}^{N,p-1}_{[Z_{\rm red}]} \to \mathcal{C}^{N,p}_{[Z_{\rm red}]}\right).
\end{equation*}
By Lemma~\ref{lmaindep2}, $\tau D\varphi_1 \cdots D\varphi_p R^E_p \tau^{-1}$ is independent of the connection $D$,
and we can thus assume that $D$ is the trivial connection $d$ in a trivialization of $E$.
Then $D\varphi_1 \cdots D\varphi_p$ is a holomorphic $\Hom(E_p,E_0)$-valued morphism, and thus, since $R^E_p$ is
a $\Hom(E_0,E_p)$-valued Coleff-Herrera current, $\tau D\varphi_1 \cdots D\varphi_p R^E_p \tau^{-1} \in \CH^p_{Z_{\rm red}}$.
Hence, $\mu \in \CH^N_{Z_{\rm red}}$, so by \eqref{eq:ch-uniqueness}, $\mu = 0$. Since $\mu = 0$ 
for any choice of the holomorphic $p$-form $\tilde{\alpha}$ on $X$, we get that \eqref{eqcm2} holds.
Finally, using \eqref{eqtraceidentification}, we get that \eqref{eqcm1} holds.
To conclude, assuming \eqref{eqfundassumption}, Theorem~\ref{thmcm} follows from the theorem
in \cite{LJ1}.

\medskip

On the other hand, Theorem~\ref{thmcm} together with \eqref{eqfundassumption} implies \eqref{eqljmaincurrent},
which in turn implies \eqref{eqljmain} since \eqref{eqdualisom2} is an isomorphism.
Thus, Lejeune-Jalabert's result follows from Theorem ~\ref{thmcm} and \eqref{eqfundassumption}.
Finally, taking Theorem~\ref{thmcm} and Lejeune-Jalabert's result for granted, it follows that
\eqref{eqfundassumption} must be a correct assumption.

\end{document}